\documentclass[11pt,reqno]{amsart}

\usepackage{amsfonts}
\usepackage{amsthm, amssymb, ulem, amscd} 
\usepackage{amsmath}
\usepackage[mathscr]{euscript}
\usepackage{color}
\usepackage{enumerate}
\usepackage{graphicx}

\def\crn#1#2{{\vcenter{\vbox{
        \hbox{\kern#2pt \vrule width.#2pt height#1pt
           }
          \hrule height.#2pt}}}}
\def\intprod{\mathchoice\crn54\crn54\crn{3.75}3\crn{2.5}2}
\def\into{\mathbin{\intprod}}

\newcommand{\stopthm}{\hfill$\square$\medskip}

\newcommand{\pa}{\partial}
\newcommand{\na}{\nabla}
\newcommand{\Ric}{\operatorname{Ric}}

\newcommand{\grad}{\operatorname{grad}}
\newcommand{\Lo}{\mathring{L}}

\newcommand{\Span}{\operatorname{span}}

\newcommand{\tr}{\operatorname{tr}}

\newcommand{\D}{\Delta}
\newcommand{\R}{\mathbb R}
\newcommand{\N}{\mathbb N}

\newcommand{\ep}{\epsilon}

\newcommand{\al}{\alpha}
\newcommand{\la}{\lambda}
\newcommand{\ga}{\gamma}

\newcommand{\de}{\delta}
\newcommand{\be}{\beta}

\newcommand{\om}{\omega}

\newcommand{\gh}{\widehat{g}}
\newcommand{\gb}{\overline{g}}
\newcommand{\hb}{\overline{h}}
\newcommand{\kb}{\overline{k}}

\newcommand{\Lb}{\overline{L}}
\newcommand{\Pb}{\overline{P}}
\newcommand{\Qb}{\overline{Q}}
\newcommand{\Gb}{\overline{\Gamma}}
\newcommand{\Jb}{\overline{\mathsf{J}}}
\newcommand{\Tb}{\overline{T}}

\newcommand{\Yb}{\overline{Y}}
\newcommand{\Ybh}{\widehat{\overline{Y}}}

\newcommand{\hh}{\widehat{h}}

\newcommand{\gt}{\widetilde{g}}

\newcommand{\Ph}{\widehat{P}}

\newcommand{\lh}{\widehat{\lambda}}

\newcommand{\wh}{\widehat}

\newcommand{\Pt}{\widetilde{P}}

\newcommand{\Tt}{\widetilde{T}}
\newcommand{\Qt}{\widetilde{Q}}

\newcommand{\gth}{\widehat{\widetilde{g}}}
\newcommand{\rt}{\widetilde{r}}
\newcommand{\xt}{\widetilde{x}}
\newcommand{\htt}{\widetilde{h}} 
\newcommand{\nt}{\widetilde{\nabla}}

\newcommand{\Gat}{\widetilde{\Gamma}}
\newcommand{\Gt}{\widetilde{\Gamma}}

\newcommand{\nb}{\overline{\nabla}}

\newcommand{\Ga}{\Gamma}

\newcommand{\Si}{\Sigma} 
\newcommand{\cL}{\mathcal{L}}
\newcommand{\cH}{\mathcal{H}}

\newcommand{\cS}{\mathcal{S}}
\newcommand{\cK}{\mathcal{K}}
\newcommand{\cA}{\mathcal{A}}
\newcommand{\cI}{\mathcal I}  
\newcommand{\cJ}{\mathcal J}  

\newcommand{\cG}{\mathcal{G}}
\newcommand{\cGt}{\widetilde{\mathcal{G}}}
\newcommand{\cGth}{\widehat{\widetilde{\mathcal{G}}}}

\newcommand{\cU}{\mathcal{U}}
\newcommand{\cV}{\mathcal{V}}
\newcommand{\cW}{\mathcal{W}}

\newcommand{\sh}{\mathrm{h}}
\newcommand{\sg}{\mathrm{g}}
\newcommand{\sF}{\mathsf{F}}
\newcommand{\sG}{\mathsf{G}}
\newcommand{\sP}{\mathsf{P}}
\newcommand{\sD}{\mathsf{D}}

\newcommand{\hhh}{{\hspace{.3mm}}}

\newcommand{\blue}[1]{\color{blue}#1~\color{black}}

\def\sideremark#1{\ifvmode\leavevmode\fi\vadjust{\vbox to0pt{\vss
 \hbox to 0pt{\hskip\hsize\hskip1em
 \vbox{\hsize2cm\tiny\raggedright\pretolerance10000
  \noindent #1\hfill}\hss}\vbox to8pt{\vfil}\vss}}}

\theoremstyle{plain}
\newtheorem{theorem}{Theorem}[section]
\newtheorem{lemma}[theorem]{Lemma}
\newtheorem{proposition}[theorem]{Proposition}
\newtheorem{corollary}[theorem]{Corollary}

\theoremstyle{definition}

\theoremstyle{remark}
\newtheorem{remark}[theorem]{Remark}

\numberwithin{equation}{section}

\title[Extrinsic GJMS Operators for Submanifolds]{Extrinsic GJMS Operators
  for Submanifolds} 

\author{Jeffrey S. Case}
\address{Department of Mathematics, Penn State University\\
University Park, PA 16802, USA}
\email{jscase@psu.edu}

\author{C Robin Graham} 
\address{Department of Mathematics, University of Washington,
Box 354350\\
Seattle, WA 98195-4350, USA}
\email{robin@math.washington.edu}

\author{Tzu-Mo Kuo}
\address{Department of Mathematics, University of California\\
Santa Cruz, CA, 95064, USA}
\email{tkuo6@ucsc.edu}

\begin{document}

\begin{abstract}
We derive 
extrinsic GJMS operators and $Q$-curvatures associated to a
submanifold of a conformal manifold.  The operators are conformally
covariant scalar differential operators on the submanifold with leading   
part a power of the Laplacian in the induced metric.  Upon realizing
the  conformal manifold  as the conformal infinity of an asymptotically 
Poincar\'e--Einstein space and the submanifold as the boundary 
of an asymptotically minimal submanifold thereof,  these  operators arise
as obstructions to 
smooth extension as eigenfunctions of the Laplacian of the induced metric 
on the minimal submanifold.  We derive explicit formulas for the operators
of orders 2 and 4.  We prove factorization formulas when 
the original submanifold is a minimal submanifold of an Einstein manifold.
We also show how to reformulate the construction in terms of the ambient 
metric for the conformal manifold, and use this to prove that the operators
defined by the factorization formulas are conformally invariant for all
orders in all dimensions.  

\end{abstract}

\maketitle

\thispagestyle{empty}

\section{Introduction}\label{intro}

The GJMS operators~\cite{GJMS} are a family of conformally covariant
natural differential operators on a Riemannian manifold with principal part
a power of the Laplacian.  They  are basic objects in conformal geometry 
which arise in many contexts.  In this paper, we construct analogs of the
GJMS operators associated to a submanifold of a conformal manifold. 
These are differential operators on the submanifold which depend on its
extrinsic geometry in the background space.

The original construction of operators in~\cite{GJMS} used the ambient
metric of~\cite{FG2}. Their derivation was reformulated 
in~\cite{GZ} in terms of a Poincar\'e metric. There they arise as 
obstructions
to smoothly extending functions on the conformal manifold
as eigenfunctions of the Laplacian of the Poincar\'e 
metric with prescribed leading order asymptotics.  
In~\cite{GZ}, it was pointed out that the same construction can be
carried out upon replacing the Poincar\'e metric by any asymptotically
hyperbolic metric.  In the general case, the operators still act on
functions on the boundary at infinity and satisfy the same conformal
transformation law with respect to rescaling the boundary metric.  However
they now depend on the Taylor coefficients at the boundary of a
compactification of the asymptotically hyperbolic metric, which in general
need not have any relation to the intrinsic geometry of the boundary
itself. 

In order to construct operators associated to a submanifold of a conformal
manifold, we apply the latter construction to an asymptotically hyperbolic
metric determined asymptotically by the extrinsic geometry of the  
submanifold.  Let $\Sigma$ be a $k$-dimensional   
submanifold of an $n$-dimensional Riemannian manifold 
$(M,g)$.  Let $X=M\times[0,\ep_0)$ for some small
$\ep_0>0$.  We realize 
the conformal class $(M,[g])$ as the conformal infinity of a Poincar\'e
metric $g_+$ on $\mathring{X}$, i.e. a smooth, even asymptotically
hyperbolic approximate solution of the Einstein equation $\Ric(g_+)=-ng_+$
(see~\cite{FG2}).  In turn we realize $\Si$ as the boundary of a smooth
submanifold $Y\subset X$ which is asymptotically minimal with respect to
$g_+$ and even in a suitable sense. 
We equip $\mathring{Y}$ with the metric $h_+$ induced by $g_+$, which is
also asymptotically hyperbolic.  We then derive our operators by applying
the usual construction of~\cite{GZ} on the space $(\mathring{Y},h_+)$.  We call the 
resulting operators {\it (minimal submanifold) extrinsic GJMS operators.}   

Branson's (critical) $Q$-curvature~\cite{B} is another fundamental object
in conformal geometry.  It is a natural scalar defined in even dimensions 
whose conformal transformation law is linear in the log of the conformal
factor.  Branson defined it from the zeroth order terms of the GJMS
operators via analytic continuation.  Just as for the operators, the same    
construction can be carried out on a general asymptotically hyperbolic
manifold.  Applying the construction on $(\mathring{Y},h_+)$, we obtain   
for $k$ even a $Q$-curvature associated to the submanifold
$\Si\subset(M,g)$.      

Our main existence result is the following.

\begin{theorem}\label{gjmsmain}
Let $(M^n,g)$ be a Riemannian manifold and $\Sigma^k\subset M$ a
submanifold, with $n\geq 3$ and $1\leq k\leq n-1$.  
\begin{enumerate}
\item
For the following values of $\ell$, there is a
minimal submanifold extrinsic GJMS operator 
$P_{2\ell}:C^\infty(\Si)\rightarrow C^\infty(\Si)$ of order $2\ell$:  
\begin{enumerate}
\item
$1\leq \ell <\infty$ if $n$ and $k$ are both odd,
\item
$1\leq \ell< n/2$ if $n$ is even and $k$ is odd,
\item \label{Ic}
$1\leq \ell\leq k/2 +1$ if $k$ is even.  (If $\ell = k/2+1$ and $n$ is 
  even, we also assume $n>k+2$.) 
\end{enumerate}
\smallskip
The operator $P_{2\ell}$ is 
formally  self-adjoint, has leading term $(-\D_h)^\ell$
where $h$ is the metric induced on $\Si$ by $g$, 
is natural (as defined in \S\ref{notation}), and satisfies the 
transformation law
\begin{equation}\label{Ptransform}
\widehat{P}_{2\ell} = e^{(-k/2-\ell)\,\om|_\Si}\circ P_{2\ell}\circ
e^{(k/2-\ell)\,\om|_\Si}\, ,
\end{equation}
under a conformal change $\gh=e^{2\om}g$ with~$\om\in C^\infty(M)$. 
\smallskip
\item
If $k$ is even, there is a minimal submanifold 
extrinsic $Q$-curvature $Q_k$, which is a natural scalar on $\Sigma$.
Under conformal change $\gh=e^{2\om}g$, it satisfies   
\begin{equation}\label{Qtransform}
e^{k\,\om|_\Si}\widehat{Q}_k = Q_k  + P_k(\om|_\Si).
\end{equation}
Moreover, $P_k1=0$.  
\end{enumerate}
\end{theorem}

Equation~\eqref{Qtransform}, the formal self-adjointness of $P_k$, and the
fact that $P_k1=0$ imply that if $\Si$ is 
compact, then $\int_\Si Q_k\,da$ is a conformal invariant.
This invariant is a multiple of the higher-dimensional Willmore energy of
$\Si$ studied in \cite{GR,Z,M}.  See Remark~\ref{intQ}.   

The transformation law \eqref{Ptransform} can be interpreted as  
saying that for fixed $\ell$, the $g$-dependent family of operators
$P_{2\ell}$ defines a single invariant operator on conformal densities on
$\Si$.
The zeroth order terms of the operators $P_{2\ell}$ for $2\ell \neq k$
define non-critical curvatures $Q_{2\ell}$ which we also study.  

In this paper, by a submanifold we mean an embedded submanifold.  But since
the construction is local and an immersed submanifold is locally embedded,
Theorem 1.1 also holds on an immersed submanifold $i:\Si\rightarrow M$.
Likewise, Theorems~\ref{factorization} and \ref{Einsteininvariance} below
apply to immersed submanifolds.   

Note that the transformation laws \eqref{Ptransform} and \eqref{Qtransform}
only involve $\om|_\Si$.  A consequence is that the $P_{2\ell}$ and 
$Q_k$ only depend on the conformal class $[g]$ on $M$ near $\Si$ and
the representative $h$ on $\Si$.  That is, they are independent of the way
that the representative $h$ is extended off of $\Si$ to a metric $g$ in the
conformal class on $M$.  This observation is pertinent to the application
of our results to immersed submanifolds in the follow-up paper \cite{CGKTW}
described below.  When we view 
the $P_{2\ell}$ and $Q_k$ as determined by $(M,[g])$ and a choice of 
representative $h$ on $\Si$, in the immersed case $h$ can be any 
representative, allowing rescalings by any positive function in
$C^\infty(\Si)$, not just the pullback of a representative $g$ on $M$. 
In this setting \eqref{Ptransform} and 
\eqref{Qtransform} hold for any $\om\in C^{\infty}(\Si)$, again not just
the pullback of a function on $M$.  

We usually write as if we are working in Riemannian signature.  But
everything in this paper is formal, so is valid for metrics $g$ of general
signature.  If $g$ has mixed signature, it is assumed that the submanifold
$\Si$ is such that $g|_{T\Si}$ is everywhere nondegenerate.  A
nondegenerate submanifold is called minimal if its mean  
curvature vector vanishes.  

The restrictions in Theorem~\ref{gjmsmain} when $n$ or $k$ is even
arise from the fact that the 
extension problems for Poincar\'e metrics and minimal submanifolds
are obstructed at finite order in these cases.  The original GJMS
construction is similarly obstructed in even dimensions.  The instance
$\ell = k/2+1$ of case~(\ref{Ic}) is subtle.   Unlike the other cases, the ambiguity
in the minimal submanifold expansion does enter into the induced metric at
the order which could affect $P_{k+2}$.  However this ambiguity does not  
contribute in the derivation of $P_{k+2}$; see \S\ref{formulas} below for  
details.  

There are other constructions of operators and $Q$-curvatures on
submanifolds of conformal manifolds.

The simplest is just to forget  
about the extrinsic geometry.  That is, consider the induced metric $h$ on  
$\Si$ and take the usual GJMS operators and $Q$-curvatures determined by
$h$.  We call these the {\it intrinsic} operators and $Q$-curvatures.
Unfortunately, in general these do not satisfy the factorization identities
given in Theorem~\ref{factorization} below, which are crucial for the
application we have in mind.
However, as we show in \S\ref{gjmssection}, our operators coincide with the
intrinsic operators for umbilic submanifolds of locally conformally flat
spaces. 

In \cite{GoW1,GoW4}, a calculus on conformally compact manifolds is
developed, including a construction of families of operators and 
$Q$-curvatures on the conformal infinity as well as generalizations
thereof to operators acting on sections of more general vector bundles.
In the setting of a hypersurface $\Si\subset (M,[g])$, 
the papers~\cite{GoW2,GoW3,GoW5,BGW} apply this general construction to a 
distinguished asymptotically hyperbolic representative of the conformal
class $(M,[g])$ to derive and study natural extrinsic operators and
$Q$-curvatures on $\Si$ analogous to those in Theorem~\ref{gjmsmain}.  This 
distinguished representative is an asymptotic solution of the singular
Yamabe problem, i.e.\ a metric which is asymptotically hyperbolic along
$\Si$ with asymptotically constant scalar curvature $-n(n+1)$.   
It follows from explicit formulas for
our operators $P_2$ and $P_4$ derived in Theorems~\ref{minimalops} and
\ref{intrinsicformulas} that our operators in the case of  
hypersurfaces are different from the Gover--Waldron singular Yamabe
extrinsic operators.  In particular, the singular Yamabe extrinsic
operators also do not satisfy the factorization identities.  Unlike our   
construction via minimal submanifold extension, the construction via the
singular Yamabe problem also produces operators of odd order.
Further studies of the singular Yamabe extrinsic operators and
$Q$-curvature can be found in \cite{JO,CMY,J}.  
The paper~\cite{AGW} uses minimal extension into
a singular Yamabe space to 
construct operators, $Q$-curvatures, and associated boundary transgression
curvatures when the hypersurface $\Si\subset M$ itself has a boundary.  

The singular Yamabe problem and therefore the 
construction of operators using it are obstructed at finite order in all
dimensions.  But in \cite{GoW3}, a tractor construction is applied to the 
highest order operator produced by the singular Yamabe construction to
produce operators of all even orders in all dimensions when $\Si$ is a
hypersurface.  This is in contrast to the situation for the original GJMS
operators, where it is known that operators of higher orders do not exist 
\cite{Gr,GH}.   It would be interesting to determine whether such operators
exist for all $\ell$ in higher codimension when $k$ and/or $n$ is even.

Suppose now that $g$ is Einstein.
In this case a  Poincar\'e metric  can be written explicitly.
If~$\Ric(g)=\la (n-1)g$, then the Schouten tensor is 
given by $\mathsf{P}_{ij}=\frac\la2 g_{ij}$, 
and a   Poincar\'e metric for~$g$ is~\cite{FG2}
\begin{equation}\label{einsteinform}
g_+ = r^{-2}(dr^2 + (1-\tfrac14 \la r^2)^2 g)
\, .
\end{equation}
We call this the {\it canonical Poincar\'e
metric} associated to the Einstein metric $g$.  This explicit identification 
of $g_+$ leads to a  
factorization formula for the GJMS operators for an Einstein metric as a
product of second order operators of the form $\D +c$.
It also leads to the conclusion that for conformal 
classes containing an Einstein metric, the operators defined by the
factorization formula are invariantly associated to the conformal class for
all $\ell \geq 1$ in all dimensions.  See~\cite{FG2}.  Another treatment of
these results is contained in \cite{Go}.  

A minimal extension $Y$ can
also be written explicitly if $\Si\subset M$ is a minimal submanifold with
respect to the Einstein metric $g$.  Namely, it was  
observed in the proof of Proposition~4.5 of~\cite{GR} 
that $Y=\Si\times [0,\ep_0)\subset X$ is a minimal extension of $\Si$ with respect
to $g_+$   (we recall this argument in \S\ref{gjmssection}). We call 
$\Si\times [0,\ep_0)$ the {\it canonical minimal extension} of the minimal
submanifold $\Si$ of the Einstein manifold $(M,g)$.   
The next theorem establishes that the same factorization formula holds for
the minimal submanifold extrinsic GJMS operators.  We regard this as a
fundamental feature of these operators.        
 
\begin{theorem}\label{factorization}
Suppose that $\Ric(g) = \la (n-1)g$ and $\Si$ is minimal in $(M,g)$.
Let $g_+$ be the canonical Poincar\'e metric~\eqref{einsteinform} and let
$Y=\Si\times [0,\ep_0)$ be the canonoical minimal extension of $\Si$.  
Denote by $h$ the metric on $\Si$ induced by $g$ and by $h_+$ the metric on
$\mathring{Y}$ induced by $g_+$.  Let $\ell\in \N$.  The operators produced
by the GJMS construction for $(\mathring{Y},h_+)$ are  
\begin{equation}\label{Pfactor} 
P_{2\ell} = \prod_{j=1}^\ell (-\D_h + \la c_j),\qquad 
c_j=(\tfrac{k}{2}+j-1)(\tfrac{k}{2}-j).
\end{equation}
If $k$ is even, then 
  \[
Q_k 
=\la^{k/2}(k-1)!
\, .  
\]
\end{theorem}

In the case that $g$ is Einstein and $\Si\subset (M,g)$ is minimal,
\eqref{Pfactor} gives a formula for the operator $P_{2\ell}$  
for $\ell$ in the ranges stated in Theorem~\ref{gjmsmain}.  The next theorem
shows that for all~$\ell\geq 1$ in all dimensions, the operators
\eqref{Pfactor} satisfy~\eqref{Ptransform} under conformal change to
another Einstein metric for which $\Si$ is also minimal (if there is
another such Einstein metric).  This result can be used to define 
consistently operators via~\eqref{Ptransform} for non-Einstein 
metrics in the same conformal class.  

\begin{theorem}\label{Einsteininvariance}
Let $(M,g)$ be Einstein with $\Ric(g) = \la (n-1)g$ and let
$\gh=e^{2\om}g$ be a conformally related Einstein metric with
$\Ric(\gh)=\widehat{\la}(n-1)\gh$.  Suppose  
$\Si\subset M$ is minimal relative to both $g$ and $\gh$.  Define
$P_{2\ell}$ by~\eqref{Pfactor}.  Let $\widehat{h}$ denote the metric
induced on $\Si$ by $\gh$ and define $\widehat{P}_{2\ell}$ by
\eqref{Pfactor} with $h$ replaced by $\widehat{h}$ and $\la$ replaced by
$\lh$.  Then~\eqref{Ptransform} is valid for all~$\ell\geq 1$.        
\end{theorem}

In this paper, asymptotically minimal submanifolds for Poincar\'e  
metrics are a tool used to derive and study the minimal submanifold
extrinsic GJMS operators $P_{2\ell}$.  The paper~\cite{FiH} is in the same
spirit.  In the case $k=1$, it uses the even asymptotically minimal
extension to study canonical parametrizations of the curve $\Si$ and to
characterize when it is a conformal geodesic.

Asymptotically hyperbolic metrics that are exact solutions of
$\Ric(g_+)=-ng_+$ are called {\it Poincar\'e--Einstein metrics}.   
Actual minimal submanifolds of Poincar\'e--Einstein spaces, not just
asymptotic ones, have been and continue to be an object of   
intense study themselves, motivated partially by physical considerations.   
In a follow-up paper \cite{CGKTW}, Theorem~\ref{factorization} together
with a construction based on scattering theory will be applied to derive a
formula of Gauss--Bonnet type for 
the renormalized area of such an even-dimensional minimal submanifold of a 
Poincar\'e--Einstein space, assuming a submanifold version of a result of
Alexakis~\cite{A} establishing a decomposition of integrands of conformally
invariant integrals.  This application was the genesis of our project:  we 
were led to search for extrinsic GJMS operators satisfying a factorization
of the form~\eqref{Pfactor} in order to derive such a formula.   
The formula takes the form
\[
\cA = a_k \chi(Y) + \int_Y \cW_k \, dv_{h_+},  
\]
where $k$ is even, $Y$ is a minimal submanifold of dimension $k$ of a
Poincar\'e--Einstein space, $h_+$ is the induced metric on $Y$,
$\chi(Y)$ denotes the Euler characteristic and 
$\cA$ the renormalized area of $Y$, $\cW_k$ is a pointwise
conformal submanifold invariant, and $a_k\in \R$.  A formula of this type
was derived for 
$k=2$ in \cite{AM} and a formula in the same spirit for $k=4$ in the case
of hypersurfaces in~\cite{T}.  The derivation in \cite{CGKTW} follows the
same outline as 
the proof in~\cite{CQY} of an analogous formula for the renormalized
volume of even-dimensional Poincar\'e--Einstein manifolds.        

The paper \cite{GZ} showed that the usual GJMS operators can be embedded in
a continuous family of 
fractional order scattering operators.  In general, such a family
depends on a choice of an exact or asymptotic Poincar\'e--Einstein manifold
with prescribed conformal infinity.  Such fractional order operators have
been of great interest in recent years, motivated in part by the connection
to the Caffarelli--Silvestre extension \cite{CS,CG,CC,CY}. 
As indicated in \cite{GZ}, the scattering construction
can be carried out for any asymptotically hyperbolic metric.  Consequently, 
our minimal submanifold extrinsic GJMS operators can also be embedded in
families of fractional order minimal submanifold extrinsic scattering
operators upon  choosing an extension of $\Si$ as an exact or
asymptotically minimal submanifold of an exact or asymptotically
Poincar\'e--Einstein space.  It would be interesting to explore possible
uses of such operators.

A summary of the paper is as follows.  In \S\ref{notation} we describe our 
notation and conventions and formulate the notion of naturality in the
submanifold setting.  In \S\ref{background}, we review the formal  
asymptotics of smooth, even Poincar\'e metrics, minimal submanifolds, and 
their induced metrics, mostly following~\cite{FG2} and~\cite{GR}.  In
\S\ref{gjmssection} we review the GJMS construction in the setting of a
general asymptotically hyperbolic metric and prove Theorem~\ref{gjmsmain}
except for case (1c) 
with $\ell=k/2 + 1$.  We discuss minimal submanifolds of Einstein
manifolds and prove Theorem~\ref{factorization}.  We then formulate  
Theorem~\ref{extensioninvariance}, which asserts the
infinite order diffeomorphism invariance of the canonical minimal
extension, and prove Theorem~\ref{Einsteininvariance} assuming
Theorem~\ref{extensioninvariance}.  We close \S\ref{gjmssection} by
showing that if $(M,[g])$ is locally conformally flat and $\Si$ is umbilic,
then the extrinsic operators equal the intrinsic operators.  
In \S\ref{formulas}, we derive two versions of formulas for $P_2$, $Q_2$
and $P_4$, $Q_4$ for general $g$, $\Si$, $k$ and $n$.  
One version is Theorem~\ref{minimalops}, in which the operators are
expressed in terms of the second fundamental form and curvature of the
background metric $g$.  This version is well-suited to 
seeing the factorization formula~\eqref{Pfactor} when $g$ is Einstein and
$\Si$ is minimal.  The
second version is Theorem~\ref{intrinsicformulas}, in which $P_2$ and $P_4$  
are written as the GJMS operators $\Pb_2$ and $\Pb_4$ intrinsic to $\Si$
plus additional terms involving extrinsic quantities.  The derivation of
both is based on a fourth-order passage   
to normal form of the induced metric $h_+$ on $Y$.  We conclude 
\S\ref{formulas} by proving Theorem~\ref{gjmsmain} in the remaining case
(1c) for $\ell = k/2+1$.  
The main step is Lemma~\ref{Udependence}, which uses a calculation of a
higher order passage 
to normal form to identify explicitly the contribution to the induced
metric of the ambiguity in the minimal submanifold expansion.
Finally, in \S\ref{ambientsection} we show how to reformulate the whole
construction in terms of the ambient metric.  We use this to prove  
Theorem~\ref{extensioninvariance}, thereby completing the
proof of Theorem~\ref{Einsteininvariance}.   The proof of 
Theorem~\ref{extensioninvariance} is modeled on the proof of the analogous
result in~\cite{FG2} asserting the diffeomorphism invariance of the
canonical Poincar\'e metric associated to an Einstein metric.   

\bigskip
\noindent
{\it Acknowledgements.}  This project was initiated at the August 2022  
workshop ``Partial differential equations and conformal geometry'' at the
American Institute of Mathematics.  The authors are
grateful to AIM for its support and for providing a structure and
environment conducive to fruitful research.  The authors 
would also like to thank Aaron Tyrrell and Andrew Waldron for their
numerous contributions to this paper.  Jeffrey S. Case was partially
supported by the Simons Foundation (Grant \#524601).

\section{Notation and Conventions}\label{notation}
For a Riemannian manifold $(M^n,g)$, we denote the Levi-Civita connection 
by ${}^g\nabla$, the curvature tensor by 
$R_{ijkl}$, the Ricci tensor by $\Ric(g)$ or $R_{ij}=R^k{}_{ikj}$, and the
scalar curvature by $R=R^i{}_i$.  Our sign convention for $R_{ijkl}$ is
such that spheres have  
positive scalar curvature.  The Schouten tensor of $(M,g)$ is 
\[
\sP_{ij}=\frac{1}{n-2}\Big(R_{ij}-\frac{R}{2(n-1)}g_{ij}\Big) 
\]
and the Weyl tensor is defined by the decomposition
\[
R_{ijkl}=W_{ijkl}+\sP_{ik}g_{jl}-\sP_{jk}g_{il}-\sP_{il}g_{jk}+\sP_{jl}g_{ik}.
\]
The Cotton and Bach tensors are
\[
C_{ijk}={}^g\nabla_k\sP_{ij}-{}^g\nabla_j\sP_{ik}
\]
and 
\[
B_{ij}={}^g\nabla^kC_{ijk}-\sP^{kl}W_{kijl}. 
\]
Latin indices $i$, $j$, $k$ run between $1$ and $n$ in local
coordinates, or can be interpreted as labels for $TM$ or its  
dual in invariant expressions such as those above (Penrose abstract index
notation).    

We will denote by $\Si$ a submanifold of $(M,g)$ of dimension $k$,
$1\leq k\leq n-1$.  All 
considerations in this paper are local, so all submanifolds are assumed 
to be embedded.  
We use $\al$, $\be$, $\ga$ as index labels for
$T\Si$ and $\al'$, $\be'$, $\ga'$ for the normal bundle $N\Si$.
A Latin index $i$ thus specializes either to an $\al$ or an $\al'$. 
So, for instance, when restricted to $\Si$, the Schouten
tensor $\sP_{ij}$ splits   
into its tangential $\sP_{\al\be}$, mixed $\sP_{\al\al'}$, and normal
$\sP_{\al'\be'}$ pieces.  Likewise, the restriction of the metric $g_{ij}$ to
$\Si$ can be identified with the metric $g_{\al\be}$ induced on $\Si$
together with the bundle  metric $g_{\al'\be'}$ induced on $N\Si$.  We 
use $g_{\al\be}$ and $g_{\al'\be'}$ and their inverses to lower and raise  
unprimed and primed indices.

The {\it second fundamental form}  $L:S^2T\Si\rightarrow N\Si$ is 
defined by $L(X,Y)=({}^g\na_X Y)^\perp$.  We typically write it as
$L_{\al\be}^{\al'}$ , or perhaps as $L_{\al\be\al'}$ or
$L_\al{}^\be{}_{\al'}$ upon lowering and/or raising indices.  Since $L$ has 
only one primed index and 
is symmetric in $\al\be$, it is not necessary to pay 
attention to the order of the three indices.  The mean curvature vector is 
$H=\tfrac{1}{k}\tr L$, i.e. the section of $N\Si$ given by    
$H^{\al'}=\tfrac{1}{k}g^{\al\be}L_{\al\be}^{\al'}=\tfrac{1}{k}L_{\al}{}^{\al\al'}$.    

The Levi-Civita connection of $g$ induces connections on $T\Si$ and  
$N\Si$ together with their duals and tensor products, all of which we
denote $\nabla$.  So, for instance, we can form the covariant
derivative $\na_\al H^{\al'}$, which is a section of $T^*\Si\otimes N\Si$.

When working in coordinates, we always use a local coordinate system 
\[
z^i=(x^\al,u^{\al'}),\quad 1\leq \al\leq k, \quad k+1\leq \al'\leq n
\]for $M$ near 
$\Si$, with the properties that  
$\Si=\{u^{\al'}=0\}$ and $\pa_\al\perp \pa_{\al'}$ on $\Si$.  We call such
a coordinate system {\it adapted}.  The coordinates $x^{\al}$ restrict to
a coordinate system on $\Si$.  On $\Si$, the vectors   
$\pa_\al$ span $T\Si$, the $\pa_{\al'}$ span $N\Si$, and the mixed 
metric components $g_{\al\al'}$ vanish.  This use of indices for
coordinates is consistent with the abstract interpretation described
above.  Partial derivatives in local
coordinates are expressed using either of the two notations
$\pa_\al u_\be = u_{\be,\al}$.  In \S\ref{ambientsection}, 
indices preceded by a semicolon, such as $u_{\be;\al}$, are used to denote 
covariant differentiation.

By a (scalar, linear) {\it natural differential operator} on
$k$-dimensional submanifolds of $n$-dimensional Riemannian manifolds, we
will mean an assignment to each $\Si^k\subset (M^n,g)$ of a differential
operator $P$ on $\Si$, such that the following two conditions hold:
\begin{enumerate}
\item
  If $\Si'\subset (M',g')$ and $\varphi: (M,g) \rightarrow (M',g')$ is an  
  isometry for which $\varphi(\Si)=\Si'$, then $\varphi^*P' = P$.
\item
  There are $m\in \N\cup \{0\}$ and universal polynomials $q_{\cI}$ such 
  that in any adapted local coordinate system $z=(x,u)$, $P$ has the form   
  \begin{equation}\label{natural}
  Pf(x) = \sum_{|\cI|\leq m} 
  q_{\cI}\left(g^{\al\be},g^{\al'\be'},\pa_z^Jg_{ij}\right)\pa_x^{\cI}f(x),\qquad 
  f\in C^\infty(\Si).  
  \end{equation}
  Here $J$ is an $n$-multiindex and $\cI$ is a $k$-multiindex.  The
  argument $\pa_z^Jg_{ij}$ denotes all derivatives of all $g_{ij}$ of
  orders up to $N$, for some $N$, except that the variables
  $\pa_x^\cJ g_{\al\al'}$ for $k$-multiindices $\cJ$ do not appear (since
  these vanish in adapted coordinates).  The $g^{\al\be}$, $g^{\al'\be'}$
  and $\pa_z^Jg_{ij}$ are evaluated at $z=(x,0)$.      
\end{enumerate}
To clarify, $q_{\cI}$ is a polynomial function on the vector space in which
the inverse metric and the metric and its derivatives  
take values in local coordinates, taking into account the symmetry in the
metric and partial derivative indices.  For instance, for $N=1$, the  
arguments are   
\[
(g^{\al\be},g^{\al'\be'},g_{\al\be},g_{\al'\be'},g_{\al\be,i},g_{\al\al',\ga'},g_{\al'\be',i}), 
\]
so each $q_{\cI}$ is a universal polynomial function  on the vector space    
\[
S^2\R^k\oplus S^2\R^{n-k}\oplus S^2{\R^k}^*\oplus S^2{\R^{n-k}}^*
\oplus (S^2{\R^k}^*\otimes {\R^n}^*)\oplus ({\R^k}^*\otimes 
(\otimes^2{\R^{n-k}}^*)) \oplus (S^2{\R^{n-k}}^*\otimes {\R^n}^*).
\]
The special case $m=0$ serves to define natural scalars of $k$-dimensional   
submanifolds of $n$-dimensional Riemannian manifolds.  In a follow-up paper
\cite{GK}, it will be shown that any natural differential operator as above
can be expressed as a linear combination of contractions of covariant
derivatives of the curvature tensor of $g$, covariant derivatives of the
second fundamental form, and covariant derivatives of $f$.  

Our sign convention for Laplacians is that $\Delta=\sum \pa_i^2$ on
Euclidean space.  Norms are always taken with respect to the metric on tensor products
induced by the metric on the underlying bundle.

\section{Background:  Smooth Even Formal Asymptotics}\label{background} 

In this section we review the formal asymptotics of Poincar\'e metrics and
minimal submanifolds thereof.  We restrict consideration here to smooth 
even expansions, which we use in \S\ref{gjmssection} to derive  
extrinsic GJMS operators.  We largely follow
\cite{FG2} for Poincar\'e metrics and~\cite{GR} for minimal submanifold
asymptotics.  The asymptotics of minimal submanifolds of
Poincar\'e--Einstein spaces have also been studied in \cite{GrW,Z,M-K}.  

Let $(M^n,[g])$ be a conformal manifold, $n\geq 2$, and $g$ a
chosen metric in the conformal class.  Set
$X=M\times [0,\ep_0)_r$, $\mathring{X}=M\times (0,\ep_0)_r$, and 
identify $M$ with $M\times \{0\}\subset X$.  
By an {\it even Poincar\'e metric}
in normal form relative to $g$, we will mean a metric $g_+$ on  
$\mathring{X}$, for some~$\ep_0>0$, of the form 
\begin{equation}\label{gnormalform}
g_+ = \frac{dr^2 + g_r}{r^2},
\end{equation}
where $g_r$ is a smooth 1-parameter family of metrics on $M$ for which
$g_0=g$, such that the Taylor expansion of $g_r$ at $r=0$ is even, and
satisfying the following:  
\begin{enumerate}
\item If $n$ is odd, then $\Ric(g_+) + ng_+$ vanishes to infinite order at
  $r=0$.
\item If $n$ is even, then $|\Ric(g_+) + ng_+|_{g_+} = O(r^n)$.  
\end{enumerate}
(For $n$ even, the definition in \cite{FG2} includes an additional trace 
condition that will not be relevant here.)  
An even Poincar\'e metric in normal form relative to $g$ exists and 
$g_r$ is unique, to infinite order if $n$ is odd, and modulo $O(r^n)$ if
$n$ is even.
Even Poincar\'e metrics in normal form relative to 
conformally related metrics are related, to infinite order if $n$ is odd,
and modulo $ O(r^n)$ if $n$ is even, by an even diffeomorphism between 
neighborhoods of $M$ in $X$ which restricts to the identity on $M$. 
See~\cite{FG2}.   
Set $\gb=r^2g_+=dr^2+g_r$.  We view $M=\pa X$ as the  
boundary at infinity relative to $g_+$.

Let $\Si\subset M$ be a smooth embedded submanifold of dimension $k$, 
$1\leq k\leq n-1$.  Let $Y^{k+1}\subset X$  
be a smooth submanifold which is transverse to $M$ and satisfies
$Y \cap M = \Si$. 
We describe $Y$ near $\Si$ in terms of a 1-parameter family of sections 
of the $g$-normal bundle $N\Si$ of $\Si$ in $M$
as follows:  The normal exponential map of $\Si$ with respect to $g$,
denoted $\exp_\Si$, defines a  
diffeomorphism from a neighborhood of the zero  
section in $N \Sigma$ to a neighborhood of $\Sigma$ in $M$.
For $r\geq 0$ small, let $Y_r\subset M$ denote the slice of $Y$ at height
$r$,  defined by $Y\cap (M\times \{r\})=Y_r\times \{r\}$.  Then $Y_r$ is a
smooth submanifold of $M$ of dimension $k$ and $Y_0=\Sigma$.  For each $r$,
there is a unique section $U_r\in\Gamma(N\Sigma)$ so that   
$\exp_\Si \{U_r(p):p\in\Si\} = Y_r$.   This defines a smooth 1-parameter
family $U_r$ of sections of $N \Sigma$ for which, near $\Si$, we have  
\begin{equation}\label{Y}
Y=\left\{\big(\exp_\Si U_r(p),r\big): p\in \Sigma, r\geq 0\right\}.
\end{equation}
In particular, $U_0=0$.  The submanifolds $Y\subset X$ that we consider 
will all be 
orthogonal to~$M$ along $\Sigma$ with respect to $\gb$.  Thus the tangent 
bundle to $Y$ along $\Sigma$ is 
$T\Sigma\oplus \Span{\pa_r}$, and the normal bundle to $Y$ along $\Si$ can
be identified with $N\Si$.  Orthogonality of $Y$ to $M$ along $\Si$ is
equivalent to the condition $\pa_r U_r|_{r=0}=0$, i.e. $U_r = O(r^2)$.

The inverse normal exponential map determines a boundary identification
diffeomorphism~$\psi$ 
from a neighborhood of $\Si$ in $Y$ to a neighborhood of $\Si$ in
$\Si\times [0,\ep_0)$ by
\[
\psi(q,r)= (\pi((\exp_\Si)^{-1}q),r),
\]
where $(q,r)\in Y\subset M\times [0,\ep_0)$ and $\pi:N\Si \rightarrow \Si$
is the projection onto the base.  It is easily seen that $\psi$ is indeed a 
diffeomorphism if $Y$ is transverse to $M$.

It is useful to realize $\psi$ explicitly in terms of geodesic
normal coordinates.  Choose a local coordinate system
$\{x^\al: 1\leq \al\leq k\}$ for an open subset $\cV\subset \Sigma$ and a
local frame 
$\{e_{\al'}(x): 1\leq \al'\leq  n-k\}$ for $N\Sigma|_{\cV}$.  Let 
$\{u^{\al'}: 1\leq \al'\leq n-k\}$ denote the    
corresponding linear coordinates on the fibers of $N \Sigma|_{\cV}$.  The
map $\exp_\Si \big(u^{\al'} e_{\al'}(x)\big)\mapsto (x,u)$ defines a  
geodesic normal coordinate system $(x^\al,u^{\al'})$ in a neighborhood
$\cW$ of $\cV$ in $M$, with respect to which $\Sigma$ is given by
$u^{\al'}=0$.  For each $(x,u)$, 
the curve $t\mapsto (x,tu)$ is a geodesic for $g$ normal to $\Sigma$.  In 
particular, in these coordinates the mixed metric components $g_{\al\al'}$
vanish on $\cV$, so that $(x,u)$ is an adapted coordinate system as defined
in \S2.  Extend the coordinates $(x,u)$ to 
$\cW\times [0,\ep_0)\subset X$   
to be constant in $r$.  In these coordinates, the diffeomorphism $\psi$ is
given by $\psi(x,u,r) = (x,r)$ for $(x,u,r)\in Y$.  The coordinates $(x,r)$
restrict to a coordinate system on $Y$.  
If $U_r$ is a 1-parameter family of sections of
$N \Sigma$ and we define $u^{\al'}(x,r)$ by
$U_r(x) = u^{\al'}(x,r)e_{\al'}(x)$, then the description~\eqref{Y} of $Y$
is the same as saying that, in the coordinates $(x,u,r)$ on $X$, $Y$ is the 
graph $u^{\al'}=u^{\al'}(x,r)$.  The notation 
$u^{\al'}(x,r)$ can therefore be interpreted as the components of $U_r$ in
the frame $e_{\al'}(x)=\pa_{\al'}$, or equivalently as the graphing
function $u^{\al'}=u^{\al'}(x,r)$ for $Y$ in these coordinates.  When we
write $U_r^{\al'}$, the index is interpreted as an abstract index
indicating that $U_r$ is a section of $N\Si$.    

We now impose the condition that $Y$ is asymptotically minimal with respect
to the metric~$g_+$.  This
becomes a system of partial differential equations on the normal vector
fields $U_r$.  Recall that minimality of $Y$ is equivalent to the statement
that the mean curvature vector field of
$Y$ with respect to~$g_+$ obeys $H_Y=0\, .$ 

\begin{proposition}\label{U}  Let $g_+$ be an even Poincar\'e metric in
normal form and $\Si$ a submanifold of $M$ as above.  
\begin{enumerate}
\item
If $k$ is odd, then there exists $U_r$ whose Taylor
expansion in $r$ at $r=0$ is even and for which $H_Y$ vanishes to
infinite order.  Such $U_r$ is unique to infinite order.  If $n$ is even,
the Taylor expansion of $U_r$ modulo $O(r^{n+2})$ is  
independent of the $O(r^n)$ ambiguity in~$g_r$.   
\item
If $k$ is even, then there exists $U_r$ so that
$|H_Y|_{\gb}=O(r^{k+2})$.  The Taylor expansion of $U_r$ modulo
$O(r^{k+2})$ is uniquely determined (and is independent of the $O(r^n)$
ambiguity in $g_r$ if $n$ is even) and is even modulo $ O(r^{k+2})$.    
\end{enumerate}
\end{proposition}

\noindent
Proposition~\ref{U} is proved in~\cite[Theorem 3.1]{GR} for $k$ even.  It 
is straightforward to verify that the same sort of analysis can be used to
prove Proposition~\ref{U} for $k$ odd.  The main point is that the minimal
submanifold equation respects parity and has indicial roots of $0$ and  
$k+2$.  The freedom at the indicial root of $0$ corresponds to the freedom
to prescribe $\Si$ arbitrarily.  When $k$ is odd, the freedom at the
indicial root of $k+2$ is fixed by requiring the expansion of $U_r$ to be
even.  When $k$ is even, the indicial root of $k+2$ generates an
obstruction to existence of a smooth solution.  Note that $U_r=O(r^2)$; a 
minimal submanifold is orthogonal to $M$ along $\Si$. 

If $\varphi$ is an even diffeomorphism that restricts to the identity on
$M$ and pulls back~$g_+$ to another even
Poincar\'e metric $\widetilde{g}_+$ in normal form relative to a 
conformally related metric, then $\varphi$ pulls back the minimal extension
$Y$ for $g_+$ to that for $\widetilde{g}_+$, to infinite order if $k$ is
odd, and modulo $ O(r^{k+2})$ if $k$ is even.  This follows from the 
isometry invariance of the minimality condition, the parity
preservation of $\varphi$, and the uniqueness of $Y$.  In this sense~$Y$ is  
conformally invariant, to infinite order if $k$ and $n$ are odd, to order
$O(r^{n+2})$ if $k$ is odd and $n$ is even, and to order $O(r^{k+2})$ if
$k$ is even.  

In case (2), the condition $|H_Y|_{\gb}=O(r^{k+2})$ only determines the
expansion of $U_r$ modulo $O(r^{k+2})$.  Here and in 
\S\ref{gjmssection} we will take the expansion to be even to infinite
order, so that the full Taylor expansion of $U_r$ is even in all cases. 
We write the expansion of $U_r$ in the form
\begin{equation}\label{Uexpand}
U_r = U_{(2)}r^2+U_{(4)}r^4+\hhh\cdots\, ,
\end{equation}
where the $U_{(2j)}$ are globally and invariantly
defined sections of $N\Sigma$ determined by the choice of metric $g$ in the
conformal class, up to the order specified by Proposition~\ref{U}.  The
first coefficient is given by $U_{(2)}=\frac12 H$, where $H$ is the
mean curvature vector of $\Si\subset M$ with respect to $g$; see (5.1) of 
\cite{GR}.    

Let $h_+$ denote the metric on $Y$ induced by $g_+$.  Since $g_+$ and $Y$
are invariant up to diffeomorphism to the orders stated above under
conformal change of $g$, it follows that $h_+$ is likewise invariant up to
diffeomorphism.  Set $\hb = r^2h_+$, 
so that $\hb$ is the metric induced by $\gb=dr^2 + g_r$. 
In terms of the coordinates $(x^\al,r)$ on $Y$ introduced above, $\hb$ is
given by
\begin{equation}\label{hbar}
\begin{aligned}
\hb_{\al\be}&=g_{\al\be}+2g_{\al'(\al} u^{\al'}{}_{,\be)}
+g_{\al'\be'}u^{\al'}{}_{,\al}u^{\be'}{}_{,\be} \, ,\\
\hb_{\al 0} &=g_{\al\al'}u^{\al'}{}_{,r}
+g_{\al'\be'}u^{\al'}{}_{,\al}u^{\be'}{}_{,r}\, ,\\ 
\hb_{00}&= 1+g_{\al'\be'}u^{\al'}{}_{,r}u^{\be'}{}_{,r}.
\end{aligned}
\end{equation}
We use a ``$0$''  index for the $r$-direction.  
The components of $\hb$ and the derivatives of $u$ are evaluated at
$(x,r)$.  The above formulas for components of $\hb$ were obtained 
from the pullback of $\gb$ upon writing
\[
g_r=g_{\al\be}(x,u,r)dx^{\al}dx^\be
+2g_{\al\al'}(x,u,r)dx^\al
du^{\al'}+g_{\al'\be'}(x,u,r)du^{\al'}du^{\be'}.
\]
In~\eqref{hbar}, all $g_{ij}$ are understood to be evaluated at
$(x,u(x,r),r)$.  Since the expansions of~$g_r$ and $u(x,r)$ are even in
$r$, it follows upon inspection of~\eqref{hbar} that the Taylor expansions
of $\hb_{\al\be}$ and 
$\hb_{00}$ in $r$ at $r=0$ are even and the Taylor expansion of 
$\hb_{\al 0}$ is odd.  
The following proposition is easily verified from~\eqref{hbar},
Proposition~\ref{U}, and the formal determination of the Poincar\'e metric  
\eqref{gnormalform}.  
\begin{proposition}\label{orders}\hfill
\begin{enumerate}
\item
  If $n$ and $k$ are both odd, then the infinite order Taylor expansions of
  $\hb_{\al\be}$, $\hb_{\al 0}$, and $\hb_{00}$ are uniquely determined by
  $\Si$ and $g$.   
\item
  If $n$ is even and $k$ is odd, then the Taylor expansions of
  $\hb_{\al\be}$ and $\hb_{00} \mod O(r^n)$ and of
  $\hb_{\al 0} \mod O(r^{n+1})$ are independent of the $O(r^n)$ ambiguity 
  in $g_r$, and therefore are uniquely determined by
  $\Si$ and $g$.   
\item
  If $k$ is even, then the Taylor expansions of
  $\hb_{\al\be}$ and $\hb_{00} \mod O(r^{k+2})$ and of
  $\hb_{\al 0} \mod O(r^{k+3})$ are
  independent of the $O(r^{k+2})$ ambiguity in $U_r$ 
  (and independent of the $O(r^n)$ ambiguity in $g_r$ if $n$ is even), and
  therefore are uniquely determined by   $\Si$ and $g$.    
\end{enumerate}
\end{proposition}

\noindent
Since $g_{\al\al'}=O(r^2)$ and $u^{\al'}=O(r^2)$, it is evident from
\eqref{hbar} that $\hb_{\al 0}= O(r^3)$ and $\hb_{00} = 1 +O(r^2)$.  In
particular, $h_+$ is asymptotically hyperbolic since 
$|dr|^2_{\hb}=1$ at $r=0$.

\section{Extrinsic GJMS Operators}\label{gjmssection}

As described in the introduction, it was noted in~\cite{GZ} that the 
GJMS construction can be carried out for general asymptotically hyperbolic
metrics.  (In this paper, asymptotically hyperbolic metrics 
have smooth compactifications.)
The conclusions of this general GJMS/$Q$-curvature construction are 
summarized in the following proposition.
We formulate the characterization of the operators as obstructions to
the existence of smooth expansions for eigenfunctions of the Laplacian of
the asymptotically hyperbolic metric, rather than by the equivalent
characterization as log coefficients in the expansions of non-smooth
solutions.  Our choice of notation is governed by our 
intended application to extrinsic GJMS operators.

\begin{proposition}\label{gjmsprop}
Let $Y^{k+1}$ be a manifold with boundary $\Si^k$, $k\geq 1$.  Let $h_+$ be    
an asymptotically hyperbolic metric on $\mathring{Y}$.  Let $h$ be a 
representative of the conformal infinity of $h_+$ and let $r$ be a defining
function for $\Si$ satisfying $r^2h_+|_{T\Si} = h$.  Let $\ell \in \N$.   
\begin{enumerate}
\item
Given
$f \in C^\infty(\Si)$, there exists $F\in C^\infty(Y)$, uniquely determined
modulo $O(r^{2\ell})$, so that $F|_\Si = f$ and $u:=r^{k/2-\ell}F$
satisfies 
\[
\left(\D_{h_+} + ((k/2)^2 - \ell^2)\right)u =  O(r^{k/2+\ell}).  
\]
The function
\begin{equation}\label{star}
\left(r^{-k/2-\ell}\left(\D_{h_+}
+ ((k/2)^2 - \ell^2)\right)u\right)\Big|_\Si  
\end{equation}
is independent of the $O(r^{2\ell})$ ambiguity in $F$, 
independent of the choice of $r$, and can be written as 
$a_\ell P_{2\ell}f$, where
$a_\ell^{-1} = (-1)^\ell2^{2(\ell-1)}(\ell-1)!^2$ and    
$P_{2\ell}$ is a formally self-adjoint differential operator on $\Si$ with
leading term $(-\D_h)^\ell$.  If $\hh = e^{2\om}h$ for
$\om \in C^\infty(\Si)$, then
\begin{equation}\label{Ptransform2}
\widehat{P}_{2\ell} = e^{(-k/2-\ell)\om}\circ P_{2\ell}\circ
e^{(k/2-\ell)\om}.
\end{equation}
\item
There is a function $Q_{2\ell}$ which depends polynomially on $k$ so that
$P_{2\ell}1 = (k/2-\ell)Q_{2\ell}$.  For $k$ even, if $\hh = e^{2\om}h$
for $\om \in C^\infty(\Si)$, then  
\[
e^{k\om}\widehat{Q}_k = Q_k  + P_k\om. 
\]
\end{enumerate}
\end{proposition}
\begin{remark}
As written, the definition of $Q_{2\ell}$ in (2) fails in the critical case
$2\ell=k$ where the factor $k/2-\ell$ vanishes.  Branson's original 
definition was by analytic 
continuation in $k$.  There are now other constructions avoiding this  
analytic continuation \cite{GZ,FG1,FeH,GP,BG, GoW1}.
\end{remark}
\begin{remark}\label{intQ}
If $k$ is even, the properties stated in Proposition~\ref{gjmsprop}
imply that $\int_\Si Q_k\,dv_h$ is conformally invariant.  This invariant
can be identified:  The volume expansion for $h_+$ reads 
\[
\operatorname{vol}_{h_+}\{r>\ep\}=c_0\ep^{-k}+c_1\ep^{1-k}+\cdots
+c_{k-1}\ep^{-1} +L\log\tfrac{1}{\ep} +O(1),
\]
where $r$ is the geodesic defining function determined by $h$.  Then
\[
\int_\Si Q_k\,dv_h = b_k L,\qquad b_k=(-1)^{k/2} 2^{k-1} (k/2)!(k/2-1)!.
\]
This is proved for asymptotically Poincar\'e--Einstein metrics
in \cite{GZ} and \cite{FG1}, and both proofs are valid for  
general asymptotically hyperbolic metrics.
The invariant $L$ was studied in \cite{GR,Z,M}
and interpreted as a conformally invariant energy of $\Si$ 
in the case that $h_+$ is the
induced metric on an asymptotically minimal 
submanifold of an asymptotically Poincar\'e--Einstein space. 
\end{remark}
\begin{remark}
The noncritical $Q$-curvatures $Q_{2\ell}$ for $\ell \neq k/2$ satisfy
the transformation law that follows from~\eqref{Ptransform2}; namely 
\[
e^{2\ell\om}\widehat{Q}_{2\ell}=Q_{2\ell} 
+(k/2-\ell)^{-1}e^{(\ell-k/2)\om}\mathring{P}_{2\ell}\left(e^{(k/2-\ell)\om}\right), 
\]
where $\mathring{P}_{2\ell}=P_{2\ell} -(k/2-\ell)Q_{2\ell}$.
\end{remark}
\begin{remark}
In the setting of Proposition~\ref{gjmsprop}, there are nonzero obstruction
operators for generic $h_+$ also for $\ell \in 1/2 +\N$.  These vanish
by parity considerations for the $h_+$ induced by even Poincar\'e metrics
on submanifolds defined by even $U_r$.  
\end{remark}
\begin{remark}\label{P2identify}
For $k=2$, one has $P_2 = -\Delta$ for any asymptotically hyperbolic metric.
\end{remark}

In Proposition~\ref{gjmsprop}, it is clear that if $\widetilde{Y}$ is a 
second manifold with boundary  
$\Si$ and $\varphi:\widetilde{Y}\rightarrow Y$ is a diffeomorphism 
that restricts to the identity on $\Si$, then for each representative $h$, 
the operators 
$P_{2\ell}$ generated by $\varphi^*h_+$ are the same as those generated by 
$h_+$.  In particular, one can take $h_+$ to be in normal form relative to
$h$.  
In the next lemma, we calculate explicitly the operators $P_2$ and
$P_4$ for a general asymptotically hyperbolic metric $h_+$ that is even and
in normal form.        
\begin{lemma}\label{generalP4}
Let 
\[
h_+ = r^{-2}\big( dr^2 + h_r\big)
\]
be an asymptotically hyperbolic metric in normal form on
$\Si\times (0,\ep_0)$, where $\Sigma$ has dimension~$k$.  Suppose  
$h_r$ has the form
\[
h_r = h + h_2 r^2 + h_4 r^4 +\blue{\cdots}.  
\]
Then
\begin{equation}\label{Psgeneral}
\begin{aligned}
P_2 &= -\Delta + \frac{k-2}{2} Q_2,\\
P_4 &= \Delta^2 +\na^\al(T_{\al\be}\na^\be) + \frac{k-4}{2} Q_{4}, 
\end{aligned}
\end{equation}
where
\begin{equation}\label{TQ}
\begin{aligned}
Q_2 &= -\tr h_{2},\\
T&= -4 h_2+(k-2)(\tr h_2) h, \phantom{\frac{k}{2}} \\ 
Q_{4} &= 8\tr h_4 +\Delta (\tr h_2) -4 |h_2|^2 +\frac{k}{2}(\tr h_2)^2.  
\end{aligned}
\end{equation}
\end{lemma}
\begin{proof}
It is useful to introduce $\rho = r^2$ since $h_+$ is even.  Then   
$h_+=\frac{d\rho^2}{4\rho^2} + \frac{\mathrm{h}_\rho}{\rho}$, where  
$\sh_\rho: = h_r = h_{\sqrt{\rho}}$.  At $\rho =0$ we have   
\begin{equation}\label{hderivs}
  \sh'= h_2,\qquad\quad \sh''= 2h_4,
\end{equation}
where $' = \pa_\rho$.  The operator $\D_{h_+}$ takes the form 
\[
\D_{h_+} = 4\rho^2\pa_\rho^2 + 2(2-k)\rho\pa_\rho +
2\rho^2\sh^{\al\be}\sh_{\al\be}'\pa_\rho +\rho\D_{\sh_\rho}.
\]
Straightforward calculation shows that
\begin{multline}\label{Oprho}
r^{\ell-k/2-2}\circ \Big(\D_{h_+}+\left((k/2)^2-\ell^2\right)\Big)\circ
r^{k/2-\ell}\\
= 4\rho \pa_\rho^2
+4\big(1-\ell +\tfrac12 \rho \sh^{\al\be}\sh_{\al\be}'\big)\pa_\rho 
+\D_{\sh_\rho} + \left(k/2-\ell\right)\sh^{\al\be}\sh_{\al\be}'.
\end{multline}
Setting $\ell =1$ and $\rho =0$ and recalling the definition of $P_2$ in 
Proposition~\ref{gjmsprop} give
$P_2 = -\D -(k/2-1)\tr h_2$ as claimed. 

When $\ell=2$, evaluating the equation
$r^{\ell-k/2-2} \Big(\D_{h_+}
+\left((k/2)^2-\ell^2\right)\Big)\big(r^{k/2-\ell}F\big)=0$
at $\rho=0$ gives
\begin{equation}\label{Fprime}
4F' = \left(\D +\frac{k-4}{2}\tr h_2\right) F.  
\end{equation}
Upon differentiating at $\rho =0$, the $\pa_\rho^2F$ terms drop out and one
sees that \eqref{star} equals 
\[
\left(\D +\frac{k}{2}\tr \sh'\right)F' 
+ \left(\left(\D_{\sh_\rho}\right)'
+\frac{k-4}{2}\left(\sh^{\al\be}\sh_{\al\be}'\right)'\right)F\, .  
\]
Now
\begin{equation}\label{somederivs}
\begin{aligned}
  \left(\D_{\sh_\rho}\right)' \;\;\;&= -(\sh')^{\al\be}\na^2_{\al\be}
-\left(\na_\be(\sh')^{\al\be}
-\tfrac12 \na^\al(\sh')^\be{}_{\be}\right)\na_\al,\\  
\left(\sh^{\al\be}\sh_{\al\be}'\right)'&= \tr \sh''-|\sh'|^2.
\end{aligned} 
\end{equation}
Substituting~\eqref{hderivs},~\eqref{Fprime},~\eqref{somederivs} and
multiplying by $a_2^{-1}=4$ give 
\begin{multline*}
P_4 = \left(\D +\frac{k}{2}\tr h_2\right)\left(\D
+\frac{k-4}{2}\tr h_2\right)\\
-4(h_2)^{\al\be}\na^2_{\al\be}
-4\left(\na_\be(h_2)^{\al\be}
-\tfrac12 \na^{\al}(h_2)^\be{}_{\be}\right)\na_\al
+2(k-4)\left(2\tr h_4 -|h_2|^2\right).
\end{multline*}
Elementary manipulations reduce this to the expression for $P_4$ written in
\eqref{Psgeneral}.   
\end{proof}

Later we will need the following lemma, which identifies the contribution
of $h_{2\ell}$ to $P_{2\ell}$.  
\begin{lemma}\label{Qtrace}
Suppose we are in the setting of Lemma~\ref{generalP4}.  Let $\ell\geq 1$.
Write $P_{2\ell}=\mathring{P}_{2\ell}+(k/2-\ell)Q_{2\ell}$.  Then 
$\mathring{P}_{2\ell}$ depends only on $h_{2j}$ for $j<\ell$, and  
\begin{equation}\label{Qdepend}
Q_{2\ell}=\ell a_\ell^{-1}\tr h_{2\ell} +\mathring{Q}_{2\ell},
\end{equation}
where $\mathring{Q}_{2\ell}$ also depends only on $h_{2j}$ for $j<\ell$.  
\end{lemma}
\begin{proof}
Let $f$, $F$ and $u$ be as in the statement of Proposition~\ref{gjmsprop}.
Introduce $\rho =r^2$ as in the proof of Lemma~\ref{generalP4}.  Equation
\eqref{Oprho} implies that
\[
\left(4\rho \pa_\rho^2
+4\big(1-\ell +\tfrac12 \rho \sh^{\al\be}\sh_{\al\be}'\big)\pa_\rho 
+\D_{\sh_\rho} + \left(k/2-\ell\right)\sh^{\al\be}\sh_{\al\be}'\right)F
=\rho^{\ell-1}G,
\]
where $G|_{\rho=0} = a_{\ell}P_{2\ell}f$.  Differentiate $\ell-1$ times  
with respect to $\rho$ and set $\rho =0$.  The right-hand side becomes 
$(\ell-1)!a_\ell P_{2\ell}f$.  
The left-hand side depends only
on $\pa_\rho^{j}F|_{\rho=0}$ for $j< \ell$ and
$\pa_\rho^j \sh_\rho|_{\rho =0}$ for $j\leq \ell$.  All the derivatives
$\pa_\rho^{j}F|_{\rho=0}$ for $j< \ell$ are determined and depend
only on $h_{2j}$ for $j<\ell$.  The only term which involves
$\pa_\rho^\ell \sh_\rho|_{\rho =0}$ arises when all the derivatives hit 
$\sh_{\al\be}'$ in the zeroth order term
$(k/2-\ell)\sh^{\al\be}\sh_{\al\be}'F$.  It follows that
$\mathring{P}_{2\ell}$ depends only on $h_{2j}$ for $j<\ell$.  Equation
\eqref{Qdepend} for $Q_{2\ell}$ follows upon equating the zeroth order term
on both sides and using that 
$\pa_\rho^\ell \sh_\rho|_{\rho =0} = \ell! h_{2\ell}$.  
\end{proof}

The original GJMS operators were defined by taking $h_+$ to be a Poincar\'e
metric for $h$.  For a Poincar\'e metric, the coefficients are given by 
\[
\begin{split}
  (h_2)_{\al\be} &= -\mathsf{P}_{\al\be}\\ 
  (h_4)_{\al\be} &=
  \frac14 \left(-\frac{B_{\al\be}}{k-4} +
  \mathsf{P}_\al{}^\ga \mathsf{P}_{\ga\be}\right),
\end{split}
\]
where the Schouten and Bach tensors refer to the metric $h$ on $\Si$.   
In this case~\eqref{Psgeneral} reduces to the formulas for the usual Yamabe and
Paneitz operators.

Now return to the setting of \S\ref{background}.  So $\Si^k$ is a
submanifold of $(M^n,g)$ with induced metric~$h$, $g_+$ is an even
Poincar\'e metric in normal form relative to $(M,g)$ on 
$\mathring{X}=M\times (0,\ep_0)$, $Y^{k+1}$ is an asymptotically minimal
extension of $\Si$ to $X$, and $h_+$ is the metric 
induced on $\mathring{Y}$ by~$g_+$.  Since~$g_+$ was chosen to be in normal
form relative to the chosen metric $g$, the
defining function~$r$ on $X$  is the geodesic defining function for~$g_+$  
determined by~$g$, and we denote also by $r$ its restriction to
$Y$.  The compactification $\hb = r^2 h_+$ is uniquely determined
to the orders stated in Proposition~\ref{orders}.

\bigskip
\noindent
{\it Proof of Theorem~\ref{gjmsmain} (except case (1c) for
  $\ell=k/2+1$)}.
We prove here 
cases (a) and (b), and case (c) with $\ell\leq k/2$.  The proof for case
(c) with $\ell=k/2+1$ will be given after Lemma~\ref{Udependence} below. 

Since $h_+$ is
asymptotically hyperbolic, we can construct the operators $P_{2\ell}$ and 
associated $Q$-curvatures according to Proposition~\ref{gjmsprop}.  These 
will depend only on the geometry of $\Si\subset (M,g)$ so long as the
numbers of derivatives applied to the components of $\hb$ are constrained 
as in the statement of Proposition~\ref{orders}.  
In this case, invariance of $h_+$ up to diffeomorphism under conformal
change of $g$ as discussed in \S\ref{background} 
implies that the resulting operators satisfy ~\eqref{Ptransform}.   

Next we show that for the stated 
ranges of $\ell$, the operators $P_{2\ell}$ of Proposition~\ref{gjmsprop} 
associated to $h_+$ are uniquely determined independently of the
ambiguities in $g_r$ and $U_r$.  If $n$ and $k$ are both odd, there are no 
ambiguities, so the operators are well-defined for all $\ell$.  This proves
case (a).  

Since $h_+ = r^{-2}\hb$, we have    
\begin{equation}\label{Laplacian}
\D_{h_+}= r^2\D_{\hb}+(1-k)r\hb^{ij}\pa_jr\,\pa_i
= r^2\hb^{ij}\left(\pa^2_{ij}-\overline{\Ga}_{ij}^m\pa_m\right)
+(1-k)r\hb^{ij}\pa_jr\,\pa_i.
\end{equation}
It follows that the differential operator
\begin{equation}\label{DO}
r^{-k/2+\ell}\circ\left(\D_{h_+}+ ((k/2)^2 -\ell^2)\right)\circ
r^{k/2-\ell}  
\end{equation}
has smooth coefficients up to $r=0$ which involve $\hb$ only through  
$\hb^{ij}$ and the Christoffel symbols $\overline{\Ga}_{ij}^m$.  

The construction of the obstruction
operator $P_{2\ell}$ involves differentiating the operator~\eqref{DO} up to 
$2\ell$ times at $r=0$.
By hypothesis, $2\ell<n$ if $n$ is even and
$2\ell<k+2$ if $k$ is even.  So acccording to Proposition~\ref{orders}, all 
derivatives of all components of $\hb$ of orders at most $2\ell$ are
determined independently of the ambiguities in $U_r$ and $g_r$.  Thus all
derivatives of the $\hb^{ij}$ terms which can enter are independent of
these ambiguities.  

The Christoffel symbols involve an extra derivative of $\hb$.  
The Christoffel symbol term in~\eqref{Laplacian} is 
$-r^2\hb^{ij}\overline{\Ga}_{ij}^m\pa_m$.  Since this includes the factor
$r^2$ and only the first order differentiation $\pa_m$, it follows that
each occurrence of $\overline{\Ga}_{ij}^m$ in the operator~\eqref{DO} is
multiplied by a factor of either $r$ or $r^2$.  So if~\eqref{DO} is
differentiated at most $2\ell$ times at 
$r=0$, at most $2\ell -1$ of the differentiations can hit
$\overline{\Ga}_{ij}^m$.  Thus again the maximum number of differentiations 
of a component of $\hb$ is $2\ell$.

We now explain why the operators $P_{2\ell}$ are natural.
The isometry invariance is a consequence of the uniqueness and 
invariance of the Poincar\'e metric and the minimal submanifold extension,
and of the GJMS construction as formulated in Proposition~\ref{gjmsprop}. 
The polynomial dependence on the inverse and the derivatives of the metric   
can be seen by following through each step of the construction as outlined 
below.

Our data is the metric $g$, written in terms of an adapted coordinate
system $(x,u)$ for $\Si$.  First consider the Poincar\'e metric for $g$.
Proposition 3.5 of \cite{FG2} asserts that each determined derivative   
$\pa_r^j g_r|_{r=0}$ of $g_r$ in \eqref{gnormalform} can be written as a
linear combination of contractions of Ricci curvature and covariant
derivatives of Ricci curvature of the initial metric $g$.  Of course, this
step is independent of $\Si$.  

Next consider the minimal submanifold $Y$.  If $(x,u)$ is an adapted 
coordinate system, then $Y$ is written as a graph $u=u(x,r)$ in $(x,u,r)$
space.  The 
graphing function $u(x,r)$ is determined by the minimal submanifold
equation, which is written explicitly in any adapted coordinate system in
(2.11) of \cite{GrW} in terms of 
the metric $g_r$ and the induced metric $\hb$ given by \eqref{hbar}.  It is 
easily verified from \eqref{hbar} that the derivatives 
$\pa_r^j \hb|_{r=0}$ of components of $\hb$ can be expressed as  
universal polynomials in derivatives of components of $g_r$ and its inverse
and derivatives of components of $u(x,r)$.  The Taylor expansion of
$u(x,r)$ is derived inductively from the minimal submanifold equation 
in terms of derivatives of $\hb$ and $g_r$ and previously
determined Taylor  coefficients of $u(x,r)$, beginning with $u(x,0)=0$.
It follows that each component of each determined derivative  
$\pa_r^j u|_{r=0}$ can be expressed as a universal polynomial in
derivatives of the initial metric $g$ and its inverse.  Hence the same is
true for the components of the determined derivatives $\pa_r^j \hb|_{r=0}$.  

Next consider the GJMS algorithm applied to the induced metric 
$h_+=r^{-2}\hb$.  This is perhaps best analyzed by first putting $h_+$ into
normal form relative to the metric $h$ on $\Si$ induced by $g$.
The passage to normal form is affected \cite[\S5]{GL}  
by solving the eikonal equation $|d\rt|^2_{\rt^2h_+}=1$ for the geodesic
defining function $\rt$ and then straightening the flow of the vector field
$\grad_{\rt^2h_+} \rt$ by solving ordinary differential equations.  The
Taylor expansions of the solutions for both steps can be inductively
calculated from the equations in terms of the expansions of the components
of $\hb$.  It therefore follows that when written in normal form 
\begin{equation}\label{newnf}
h_+ = \rt^{-2}\left(d\rt^2 + \htt_{\rt}\right),
\end{equation}
the components of each determined derivative $\pa_r^j\htt_{\rt}|_{\rt=0}$
can be written as universal polynomials in
derivatives of the initial metric $g$ and its inverse.  See
Lemmas~\ref{nflemma} and \ref{Udependence} below for some explicit special
cases.   

Finally, when applied to a metric in normal form \eqref{newnf}, the GJMS
algorithm produces operators whose coefficients can be written polynomially
in terms of derivatives of the Taylor coefficients of $\htt$, as in
Lemmas~\ref{generalP4} and \ref{Qtrace}.  
\stopthm

As described in the Introduction, if $g$ is Einstein, then a Poincar\'e
metric is given by~\eqref{einsteinform}.  If in addition $\Si\subset (M,g)$
is minimal, then $Y=\Si\times [0,\ep_0)$ is a minimal extension.  
We recall the argument that $Y$ is minimal for completeness.  Upon setting 
$s=-\log r$,  the metric $g_+$ takes the form   
\[
g_+ = ds^2 + (e^s - \tfrac14 \la e^{-s})^2g.
\]
But it is a general fact that if $\Si$ is a minimal submanifold of a
Riemannian manifold $(M,g)$, then $\Si\times \R$ is a minimal submanifold
of $M\times \R$ with respect to any warped product metric $g_+$ of the form
$g_+ = ds^2 + A(s)g$, where $s$ denotes the variable in $\R$ and $A(s)$ is
a positive function.

\bigskip
\noindent
{\it Proof of Theorem~\ref{factorization}.}  
Recall that $h_+$ is defined to be the pull back of $g_+$ given by
\eqref{einsteinform} to $Y=\Si\times [0,\ep_0)$.  Clearly this gives 
\begin{equation}\label{hform}
h_+ = r^{-2}\big(dr^2 + (1-\tfrac14 \la r^2)^2 h\big).
\end{equation}
If $g$ is an Einstein metric on $M^n$ satisfying $\Ric(g) = \la (n-1)g$,  
then its usual GJMS operators are given by the same formula~\eqref{Pfactor}
with $\D_h$ replaced by $\D_g$ and $k$ replaced by $n$.  There are now
several proofs of this fact.  In Chapter 7 of~\cite{FG2}, it is shown 
that in the Einstein case, the recursion for the GJMS
operators reduces to a recursion for a sequence of polynomials of the one 
variable $\D_g/\la$.  The last proof presented there solves this recursion
explicitly to obtain the factorization.  This proof carries over verbatim
for the operators determined by the asymptotically hyperbolic 
metric $h_+$ given 
by~\eqref{hform}.  The relevant point is that in~\eqref{einsteinform}, the
metric $g$ is independent of $r$.  Since $h_+$ in~\eqref{hform} is given
by the same formula with $h$ independent of $r$, the same proof applies.   
This proves~\eqref{Pfactor}.  The formula for $Q$ follows upon 
inspecting the constant term in~\eqref{Pfactor}.
\stopthm

\begin{remark}
Inspection of the constant term in the factorization~\eqref{Pfactor} shows
that the $Q_{2\ell}$ for general $\ell\geq 1$ are given by 
\[
Q_{2\ell}= \la^\ell \hhh \prod_{j=1}^{2\ell-1}\big(\tfrac{k}{2}-\ell +j\big)  
\, .
\]
\end{remark}

If $n$ is odd, then~\eqref{einsteinform} is the unique even Poincar\'e
metric in normal form for $g$ to infinite
order.  But if $n$ is even, there are others, differing at order $n$.
Proposition 7.5 of~\cite{FG2} shows that the canonical Poincar\'e
metric $g_+$ defined by~\eqref{einsteinform} is conformally invariant in
the sense that if $\gh=e^{2\om}g$ is a conformally related Einstein metric
with $\Ric(\gh)=\widehat{\la}(n-1)\gh$, then $g_+$ and 
$\gh_+=r^{-2}\big(dr^2 + (1-\frac14\widehat{\la} r^2)^2\gh\big)$ are
diffeomorphic to infinite order by a diffeomorphism that restricts to the
identity on $M$.  The diffeomorphism is unique to infinite order, since the
diffeomorphism putting any asymptotically hyperbolic metric into normal
form relative to a given representative for the conformal infinity is
unique.    

If $k$ is odd, then $Y=\Si\times [0,\ep_0)$ is to infinite order the unique 
even minimal extension.  If $k$ is even, the fact that $Y$ is smooth
implies that the obstruction to existence of a  
smooth minimal extension of $\Si$ vanishes \cite[Proposition 4.5]{GR}.
But there are other infinite order minimal extensions, corresponding to  
different choices of the freedom in $U_r$ at order $k+2$.  The following
theorem shows that the canonical minimal extension is conformally
invariant.

\begin{theorem}\label{extensioninvariance}
Let $(M,g)$ be Einstein and let $\gh=e^{2\om}g$ be a conformally related
Einstein metric.  Suppose $\Si\subset M$ is minimal relative to both $g$
and $\gh$.  Then the diffeomorphism that pulls back $\gh_+$ to $g_+$
preserves $\Si\times [0,\ep_0)$ to infinite order.   
\end{theorem}

\noindent
For $k$ odd, this follows from the uniqueness to infinite order of the even   
minimal extension.  We will prove Theorem~\ref{extensioninvariance} for $k$
even in \S\ref{ambientsection} using the realization of the minimal
extension in the ambient space.  So for  
minimal submanifolds $\Si$ of Einstein manifolds $(M,g)$, for all $k$ and
$n$ the canonical choices of both $g_+$ and $Y$ are invariant to 
infinite order up to diffeomorphism.  

\bigskip
\noindent
{\it Proof of Theorem~\ref{Einsteininvariance}.}
Theorem~\ref{extensioninvariance} implies that the metrics $h_+$ and
$\widehat{h}_+$ induced on
$\Si\times [0,\ep_0)$ by $g_+$ and $\gh_+$ are diffeomorphic to 
infinite order.  The diffeomorphism invariance in
Proposition~\ref{gjmsprop} shows that the operators $P_{2\ell}$ and
$\Ph_{2\ell}$ arising from the GJMS construction for $h_+$ and
$\widehat{h}_+$ satisfy~\eqref{Ptransform}.  
Theorem~\ref{factorization} shows that these operators are given by 
the factorization formula~\eqref{Pfactor}.  
\stopthm

We close this section by identifying the extrinsic operators in another 
special case: $g$ is locally conformally flat and $\Si$ is umbilic.  Any 
locally conformally flat manifold $(M,[g])$, of any dimension $n\geq 3$, is
the conformal infinity of a hyperbolic metric $g_+$ in a  
deleted neighborhood of $M$ in $X=M\times [0,\ep)$; see \cite[Proposition
7.2, Theorem 7.4]{FG2}.   The hyperbolic metric $g_+$ is uniquely
determined by $(M,[g])$ up to a diffeomorphism restricting to the identity
on $M$.  There is an explicit formula for $g_+$ when written in normal form
relative to any representative $g$, but we will not need this formula.   

Suppose $\Si\subset (M,[g])$ is umbilic of dimension $k$,
$2\leq k\leq n-1$.  (We rule out the case $k=1$ because the umbilic
condition is vacuous in this case.)  We claim that there is an extension   
$Y$ in a neighborhood of $\Si$ in $X$ which is totally geodesic with 
respect to $g_+$, and such a $Y$ is unique.
To see this, choose locally a representative $g_0 \in [g]$ of constant
sectional curvature one on an open set $\cU\subset M$ and an isometric
embedding of $(\cU,g_0)$ into the round $n$-sphere.  Use this to regard 
$\cU\subset S^n=\pa B^{n+1}$.  
Now $X$ can be realized locally as a neighborhood of $\cU$ in $B^{n+1}$ and
$g_+$ as the standard hyperbolic metric  
$\frac{4}{(1-\lvert x\rvert^2)^2}\lvert dx\rvert^2$.  A connected
component of $\Si\cap \cU$ is a piece of a sphere
$S^k \subset \cU \subset S^n$.  The sphere $S^k$ extends into $B^{n+1}$ as
a spherical cap intersecting the   
boundary orthogonally. This extension is the unique totally geodesic
extension into the hyperbolic ball.      

Since the hyperbolic metric $g_+$ and the totally geodesic extension $Y$
are uniquely determined by $(M,[g])$ and $\Si$ up to diffeomorphism, the
induced metric $h_+$ on $Y$ is too.  So the GJMS construction on $(Y,h_+)$ 
produces extrinsic operators $P_{2\ell}$ satisfying the conformal
covariance relation \eqref{Ptransform} for all $\ell\geq 1$ and for all
$n\geq 3$ and $k\geq 2$.   

Consider now the intrinsic GJMS operators for the induced conformal class
on $\Si$, which we denote by $\Pb_{2\ell}$.  We have to rule out $k=1$ here
too since intrinsic operators don't exist on 1-manifolds.  For $k=2$, the
only intrinsic operator is $\Pb_2=-\Delta = P_2$ (recall
Remark~\ref{P2identify}).  So suppose $k\geq 3$.  If as above   
we choose locally a round representative $g_0$,   
then since $\Si\cap \cU$ is a piece of a sphere, the 
metric induced by $g_0$ on $\Si\cap \cU$ is Einstein.  So the induced
conformal class contains (local) Einstein representatives.  As mentioned in
the Introduction and discussed in \cite[Chapter 7]{FG2}, for such a
conformal class, operators can be defined for 
all $\ell\geq 1$ for all representatives so that the conformal covariance 
relation holds.  For a local Einstein representative, the operator
$\Pb_{2\ell}$ is given by the factorization formula \eqref{Pfactor} with
$P_{2\ell}$ on the left-hand side replaced by $\Pb_{2\ell}$.  

\begin{proposition}
Let $g$ be locally conformally flat and $\Si\subset (M,g)$ be umbilic of
dimension $k\geq 3$.  For any $\ell\geq 1$, the  
minimal submanifold extrinsic operator $P_{2\ell}$ equals the intrinsic
operator $\Pb_{2\ell}$.  
\end{proposition}
\begin{proof}
Since both operators satisfy the same covariance 
relation it suffices to prove the equality for a single representative, and
it suffices to prove it locally.  Take $g_0$ as above.  The 
extrinsic operator $P_{2\ell}$ is constructed via the GJMS algorithm
on $(Y,h_+)$.  Since $Y$ is a piece of a totally geodesic sphere, the 
induced metric $h_+$ is hyperbolic and thus is the  canonical 
Poincar\'e metric associated to the Einstein metric induced by $g_0$ on
$\Si$.  So $P_{2\ell}$ and $\Pb_{2\ell}$ are determined by the same
construction.
\end{proof}

\section{$P_2$ and $P_4$}\label{formulas}

We now turn to the derivation of formulas for the extrinsic GJMS operators
$P_2$ and $P_4$.  We intend to use Lemma~\ref{generalP4} for this purpose.
But $h_+$ is not in normal form in the boundary identification
$\psi:Y \rightarrow \Si\times [0,\ep_0)$   
introduced in \S\ref{background}.  This is clear from~\eqref{hbar}:  the
equations $\hb_{\al 0}=0$, $\hb_{00}=1$ need not hold.  The defining
function $r$ on $Y$ is the restriction of a geodesic defining function
for $g_+$ on $X$ but need not be a geodesic defining function for $h_+$.
We must first rewrite $h_+$ in normal form in order to apply
Lemma~\ref{generalP4}.   

By Proposition~\ref{orders} and the statement immediately thereafter, the
beginning terms of $\hb$ in~\eqref{hbar} can be written in the following
form:  
\begin{equation}\label{hbcoeffs}
\begin{aligned}
\hb_{\al\be}&=h_{\al\be}+D_{\al\be}r^2+K_{\al\be}r^4+o(r^4)\, ,\\ 
\hb_{\al 0}&= A_{\al}r^3+o(r^3)\, ,\\
\hb_{00}&=1+Er^2 +F r^4+o(r^4)\, ,
\end{aligned}
\end{equation}
relative to the boundary identification $\psi$, for some coefficients 
$D_{\al\be}$, $K_{\al\be}$, $A_{\al}$, $E$, $F$ defined on $\Si$.

\begin{lemma}\label{nflemma}
  Let $h_+=r^{-2}\hb$ be an asymptotically hyperbolic metric on
  $\Si\times (0,\ep_0)_r$ with Taylor    
coefficients given by~\eqref{hbcoeffs}.  When written in normal form 
\eqref{newnf} with respect to the same representative $h$ on $\Si$,  
the Taylor expansion of $\htt_{\rt}$ has the form 
\begin{equation}\label{newh}
\htt_{\rt} =h+\htt_2\rt^2+\htt_4\rt^4+\blue{\cdots},
\end{equation}
with
\begin{equation}\label{newcoeffs}
\begin{aligned}
  (\htt_2)_{\al\be}&=D_{\al\be}+\tfrac12 E h_{\al\be},\\
  (\htt_4)_{\al\be}&=K_{\al\be}-\tfrac12 \na_{(\al}A_{\be)}
  +\tfrac18 \na^2_{\al\be}E +\left(\tfrac14 F
  -\tfrac{3}{16}E^2\right)h_{\al\be}.  
\end{aligned}
\end{equation}
\end{lemma}
\begin{proof}
The existence result for the normal form states that there exists a  
unique change of variables
$x=\xt(x,r)$, $r=\rt(x,r)$  with $\xt(x,0)=x$, $\rt=r+O(r^2)$,  
such that $h_+$ is in normal form~\eqref{newnf} relative to $(\xt,\rt)$.  
It is tedious but straightforward to verify that
$h_+$ takes the claimed form modulo higher order terms under the change  
\[
\begin{aligned}
x^{\al}&=\xt^{\al} +\tfrac14\left(-A^\al +\tfrac14 \na^\al
E\right)\rt^4,\\ 
r&=\rt -\tfrac14 E\,\, \rt^3 +\left(-\tfrac18 F+\tfrac{3}{16} 
E^2\right)\rt^5.
\end{aligned}
\]
In carrying out this verification, note that~\eqref{hbcoeffs} refers to the
coefficients in
\[
\hb = \hb_{\al\be}(x)dx^\al dx^\be + 2\hb_{\al 0}(x)dx^\al dr
+ \hb_{00}(x)dr^2,
\]
i.e. $h_{\al\be}$ and the coefficients of the powers of $r$ in   
\eqref{hbcoeffs} are evaluated at $x$.  Similarly, in~\eqref{newh} one has
\[
\htt_{\rt}
=
\left(h_{\al\be}(\xt)+(\htt_2)_{\al\be}(\xt)\rt^2
+(\htt_4)_{\al\be}(\xt)\rt^4+\cdots\right)d\xt^\al d\xt^\be,
\]
i.e. $h$ in~\eqref{newh} and the coefficients in~\eqref{newcoeffs}
are evaluated at $\xt$.   
\end{proof}

Observe that $h_4$ enters in~\eqref{TQ} only via its trace.  Equation  
\eqref{newcoeffs} gives
\begin{equation}\label{traceh4}
  \tr \htt_4 = K_\al{}^\al -\tfrac12 \na^\al A_\al +\tfrac18 \D E
  +\tfrac{k}{4} \left(F-\tfrac34 E^2\right).  
\end{equation}

The coefficients $D_{\al\be}$, $K_{\al\be}$, $E$, and $F$ in
\eqref{hbcoeffs} were identified in   
(5.15), (5.16) of~\cite{GR} for the metrics $h_+$ that are induced on
minimal submanifolds $Y$ by Poincar\'e metrics $g_+$.
(In~\cite{GR},~$K_{\al\be}$ was called $Q_{\al\be}$, and the convention
$H^{\al'}=L_{\al}{}^{\al\al'}$ was used.)  These coefficients were   
calculated in~\cite{GR} by first evaluating the coefficients $U_{(2)}$,
$U_{(4)}$ in the expansion~\eqref{Uexpand}  
of the normal graph $U_r$, then Taylor expanding~\eqref{hbar} 
and inserting the expansion of $U_r$ and the Poincar\'e metric 
expansion of $g_r$.    
The same process gives $A_\al$; only the leading term for
$\hb_{\al 0}$ in~\eqref{hbar} is needed for this.  The formulas are:  
\begin{equation}\label{minimalcoeffs}
\begin{split}
D_{\al\be}&=-H^{\al'}L_{\al\be\al'}-\mathsf{P}_{\al\be},\\
K_{\al\be}&=-2L_{\al\be\al'}U_{(4)}^{\al'}
+\tfrac{1}{4} R_{\al'\al\be\be'}H^{\al'}H^{\be'}
+\tfrac{1}{4} L_{\al\ga}^{\al'}L_{\be}{}^\ga{}_{\be'}H_{\al'}H^{\be'}
-\tfrac{1}{2}{}^g\na_{\al'}\mathsf{P}_{\al\be}H^{\al'}\\
&\quad +L_{\ga(\al}^{\al'}\mathsf{P}_{\be)}^{\vphantom{\al'}}{}^\ga H_{\al'}   
+\tfrac{1}{4(4-n)} B_{\al\be}+\tfrac14 \mathsf{P}_\al{}^i\mathsf{P}_{i\be}
-\mathsf{P}_{\al'(\al} \na_{\be)} H^{\al'}
+\tfrac{1}{4}\na_\al H^{\al'}\na_{\be}H_{\al'},\\ 
A_\al&=-\mathsf{P}_{\al\al'}H^{\al'}+\tfrac{1}{2}H_{\al'}\na_\al H^{\al'},\\
E&= |H|^2,\\
F&=-\mathsf{P}_{\al'\be'}H^{\al'}H^{\be'}
+8 H_{\al'}U_{(4)}^{\al'}.
\end{split}
\end{equation}
The $U_{(4)}$ expression appearing in $K_{\al\be}$ and $F$ is defined 
in~\eqref{Uexpand}.  Its explicit form is given in Proposition 5.5 of
\cite{GR}, but will not be needed here since the occurrences in
$K_\al{}^\al$ and $F$ cancel in~\eqref{traceh4}.     

\begin{theorem}\label{minimalops}
The first two extrinsic GJMS operators $P_2$, $P_4$ are given by
\begin{equation}\label{Ps}
\begin{aligned}
P_2 &= -\Delta + \frac{k-2}{2} Q_2,\\
P_4 &= \Delta^2 +\na^\al(T_{\al\be}\na^\be) + \frac{k-4}{2} Q_{4}, 
\end{aligned}
\end{equation}
where
\begin{equation}\label{TQmin}
\begin{split}
Q_2 &= \mathsf{P}_\al{}^\al + \frac{k}{2}|H|^2,\\
T_{\al\be}&= 4\mathsf{P}_{\al\be} +4 H^{\al'}L_{\al\be\al'} 
-\left[(k-2)\mathsf{P}_\ga{}^\ga 
  +\frac12 (k^2-2k+4)|H|^2\right]h_{\al\be},\\
Q_4 &= -\D\left(\mathsf{P}_\al{}^\al+\frac{k}{2}|H|^2\right)
-2\Big|\mathsf{P}_{\al\be}+ H^{\al'}L_{\al\be\al'}-\frac{1}{2}|H|^2h_{\al\be}\Big|^2\\
&\quad +2\Big|\mathsf{P}_{\al\al'}-\na_\al H_{\al'}\Big|^2
+\frac{k}{2}\left(\mathsf{P}_\al{}^\al+\frac{k}{2}|H|^2\right)^2\\ 
&\quad -2 W^\al{}_{\al'\al\be'}H^{\al'}H^{\be'}
-4 C^{\al}{}_{\al\al'}H^{\al'} -\frac{2}{n-4} B^\al{}_{\al}.  
\end{split}
\end{equation}
In these formulas, $L$ and $H$ denote the second fundamental form and mean
curvature vector of $\Si\subset M$, all curvature quantities are of the
background metric $g$, $\na_\al$ denotes the induced connection on $T\Si$
and $N\Si$, and $\D=\na^{\al}\na_{\al}$.  
\end{theorem}

\begin{proof}
Substitution of~\eqref{minimalcoeffs} into~\eqref{newcoeffs} gives
\begin{equation}\label{h2}
(\htt_2)_{\al\be}=
-\left(\mathsf{P}_{\al\be}+ H^{\al'}L_{\al\be\al'}
-\frac{1}{2}|H|^2h_{\al\be}\right).
\end{equation}
Taking the trace gives
\begin{equation}\label{traceh2}
\tr \htt_2 = -\left(\mathsf{P}_\al{}^\al+\frac{k}{2}|H|^2\right).
\end{equation}
Substituting these into the first two formulas of~\eqref{TQ} with $h_2$
replaced by $\htt_2$ gives immediately the formulas for $Q_2$ and
$T_{\al\be}$ in~\eqref{TQmin}.   

We claim next that
\begin{equation}\label{traceh4again}
8\tr \htt_4 = 2\Big|\mathsf{P}_{\al\al'}-\na_\al H_{\al'}\Big|^2
+2\big|\htt_2\big|^2-2 W^\al{}_{\al'\al\be'}H^{\al'}H^{\be'}
-4C^{\al}{}_{\al\al'}H^{\al'} -\frac{2}{n-4} B^\al{}_{\al}.  
\end{equation}
This can be verified by direct computation upon substituting
\eqref{minimalcoeffs} into~\eqref{traceh4}.  
All of the computations are straightforward except
for the following point.  
Note from~\eqref{traceh4} that $\tr \htt_4$
includes the term $-\frac12 \na^\al A_\al$.  Clearly
\[
\na^\al A_\al = - H^{\al'}\na^\al \mathsf{P}_{\al\al'} 
-\mathsf{P}_{\al\al'} \na^\al H^{\al'}
+\tfrac{1}{2}|\na H|^2+\tfrac{1}{2}H_{\al'}\D H^{\al'}. 
\]
In the first term, the $\na^\al$ denotes the induced connection on the
bundle $T\Si\otimes N\Si$.  One rewrites this in terms of 
the Levi-Civita connection ${}^g\na$ using the relation    
\[
\na_\al \mathsf{P}_{\be\al'} = {}^g \na_\al \mathsf{P}_{\be\al'}
+L^{\be'}_{\al\be}\mathsf{P}_{\be'\al'} -L^\ga_{\al\al'}\mathsf{P}_{\be\ga},
\]
where the first term on the right-hand side is interpreted as the
$\al\be\al'$ 
component of ${}^g\nabla \sP$.  This relation is a consequence of the
definitions of the second fundamental form and the induced connection.  See
(2.2) of~\cite{GR} and the discussion there for elaboration of this point.   

Now substituting~\eqref{h2},~\eqref{traceh2}, and~\eqref{traceh4again} into
\eqref{TQ} gives the formula for $Q_4$ in~\eqref{TQmin}.  
\end{proof}

The derivation of Theorem~\ref{minimalops} automatically produces
formulas~\eqref{Ps} for $P_2$ and $P_4$ written in terms
of curvature of the background metric $g$.  This is advantageous for
verifying the factorization formula of Theorem~\ref{factorization} when $g$
is Einstein and $\Si$ is minimal, since in that case all the terms
involving the second fundamental form drop out, the Schouten tensor is
constant, and the Bach tensor vanishes.  This 
approach also avoids application of the Gauss--Codazzi equations.   
However, it can also be useful to express the operators in terms of
intrinsic geometry, for instance to compare with the intrinsic operators on
$\Si$ and thereby to verify directly the conformal 
transformation law~\eqref{Ptransform}.
This cannot be done in all cases, though.  
Intrinsic operators do not exist when $k=1$, and intrinsic $P_4$ 
does not exist when $k=2$.   

In the rest of this section, we will denote intrinsic quantities for the
induced metric~$h$ on~$\Si$ with an overline.  For instance,
$\overline{\mathsf{P}}_{\al\be}$ is the Schouten tensor and $\overline{R}$
the scalar curvature of $h$.  Set $\Jb:= \frac{\overline{R}}{2(k-1)}$, so
that $\Jb=\overline{\mathsf{P}}_\al{}^\al$ if $k>2$.  
Recall that our convention is that $\na_{\al}$ denotes the induced
connection on $T\Si$ or $N\Si$, so we do not need the overline on
$\D=\na^\al\na_\al$.  Thus the intrinsic Yamabe and Paneitz operators on
$\Si$ are written
\[
\begin{aligned}
  \Pb_2&=-\D +\frac{k-2}{2}\Jb,\\
  \Pb_4&=\D^2
  +\na^\al\left(\Tb_{\al\be}\na^\be\right)
  +\frac{k-4}{2}\Qb_4,
\end{aligned}
\]
where
\[
\begin{aligned}
\Tb_{\al\be}&=4\overline{\mathsf{P}}_{\al\be}-(k-2)\Jb h_{\al\be},\\ 
\Qb_4&=-\D\Jb-2|\overline{\mathsf{P}}|^2 +\frac{k}{2}\Jb^2. 
\end{aligned}
\]

Extrinsic geometry enters via the following three objects.  
The (manifestly conformally invariant) Fialkow tensor \cite{F} is 
\[
\sF_{\al\be} : =
\frac{1}{k-2}\left(\Lo_{\al\ga\al'}\Lo_\be{}^{\ga\al'}-W_{\al\ga\be}{}^\ga
-\sG h_{\al\be}\right),
\]
where
\[
\sG:=\frac{1}{2(k-1)}\left(|\Lo|^2-W_{\al\be}{}^{\al\be}\right) =
\sF_{\al}{}^{\al}. 
\]
Set
\[
\sD_{\al\al'}:=\frac{1}{1-k}\left(\na^\be\Lo_{\al\be\al'}+W_{\al\be\al'}{}^\be\right).   
\]
\begin{theorem}\label{intrinsicformulas}
The extrinsic operators $P_2$, $P_4$ of Theorem~\ref{minimalops} 
can be written
\begin{align}
P_2 &= \Pb_2 + \frac{k-2}{2}\sG, & k>1 \label{P2extint},\\
P_4 &= \Pb_4 + \na^\al(\Tt_{\al\be}\na^\be) +
\frac{k-4}{2}\Qt_4, & k>2,\label{P4extint} 
\end{align}
where
\begin{equation}\label{TQtilde}
\begin{aligned}
  \Tt_{\al\be} &= 4\sF_{\al\be}-(k-2)\sG h_{\al\be},\\ 
\Qt_4 &= -\D \sG -2|\sF|^2 +\frac{k}{2}\sG^2
-4\sF_{\al\be}\overline{\mathsf{P}}^{\al\be} +k\sG \Jb +2|\sD|^2\\    
&\hspace{.2in}- 2 H^{\al'}H^{\be'} W^{\al}{}_{\al'\al\be'}
- 4 H^{\al'}C^\al{}_{\al\al'}
-\frac{2}{n-4}B_\al{}^\al.
\end{aligned}
\end{equation}
In particular,
\begin{equation}\label{TQbartilde}
\begin{aligned}
Q_2&= \Jb+\sG,\\    
T_{\al\be} &= \Tb_{\al\be}+\Tt_{\al\be},\\
Q_4 &= \Qb_4 + \Qt_4.  
\end{aligned}
\end{equation}
\end{theorem}
\begin{proof}
The Gauss--Codazzi equations have as a consequence
the following relations (see~\cite{F}):  
\begin{align}
\mathsf{P}_{\al\be}+ H^{\al'}L_{\al\be\al'}-\frac{1}{2}|H|^2h_{\al\be} 
&=\overline{\mathsf{P}}_{\al\be} +\sF_{\al\be}, \label{GC1}\\
\mathsf{P}_{\al\al'}-\na_\al H_{\al'}&=\sD_{\al\al'}\label{GC2}.
\end{align}
The trace of the first equation gives
\begin{equation}\label{GCtrace}
\mathsf{P}_\al{}^\al+\frac k2 \hhh |H|^2 = \Jb +\sG.
\end{equation}

Combining~\eqref{GCtrace} with~\eqref{Ps} and \eqref{TQmin} shows that  
\[
P_2 = -\D +\frac{k-2}{2}Q_2 = -\D +\frac{k-2}{2}(\Jb + \sG)
=\Pb_2 +\frac{k-2}{2}\sG.
\]
This proves~\eqref{P2extint}.

For~\eqref{P4extint}, it suffices to show that $T_{\al\be}$ and $Q_4$ in
\eqref{TQmin} satisfy~\eqref{TQbartilde}.  Using~\eqref{GC1} and 
\eqref{GCtrace} in~\eqref{TQmin} gives 
\[
\begin{split}
T_{\al\be}&=4\Big(\overline{\mathsf{P}}_{\al\be}+\sF_{\al\be} 
+\frac{1}{2}|H|^2h_{\al\be}\Big) 
-\left((k-2)\big(\Jb+\sG-\frac{k}{2}|H|^2\big)
  +\frac12 (k^2-2k+4) |H|^2\right)h_{\al\be}\\
  &= 4\big(\overline{\mathsf{P}}_{\al\be}+\sF_{\al\be}\big)
-(k-2)(\Jb+\sG)h_{\al\be}\\
&=\Tb_{\al\be} +\Tt_{\al\be}.
\end{split}
\]

Finally, substituting~\eqref{GC1},~\eqref{GC2}, and \eqref{GCtrace} in  
\eqref{TQmin} gives
\[
\begin{split}
  Q_4 &= -\D(\Jb+\sG)-2|\overline{\mathsf{P}}+\sF|^2
  +2|\sD|^2 +\frac{k}{2}(\Jb+\sG)^2\\ 
&
\qquad\qquad\qquad
-2 W^\al{}_{\al'\al\be'}H^{\al'}H^{\be'}
-4 C^{\al}{}_{\al\al'}H^{\al'} -\frac{2}{n-4} B^\al{}_{\al}\, .
\end{split}
\]
Expanding and collecting terms, one sees that this equals  
$\Qb_4+\Qt_4$.  
\end{proof}
\begin{remark}
If we set
$\Pt_4 = P_4-\Pb_4=\na^\al(\Tt_{\al\be}\na^\be) + \frac{k-4}{2}\Qt_4$,
it follows from the conformal covariance of $P_4$ and $\Pb_4$ that  
under a conformal change $\gh=e^{2\om}g$ on $M$, we have 
\begin{equation}\label{tildetransform}
\begin{aligned}
\widehat{\Pt}_4 &= e^{(-k/2-2)\,\om|_\Si}\circ \Pt_4\circ 
e^{(k/2-2)\,\om|_\Si}, & k & \geq 3,\\
e^{4\,\om|_\Si}\widehat{\Qt}_4 &= \Qt_4  + \Pt_4(\om|_\Si), & k & =4.  
\end{aligned}
\end{equation}
Reversing the logic, these conformal transformation laws can be
checked directly and thereby used to verify by direct calculation the
conformal covariance of $P_4$.

Note that the second and third terms $-2|\sF|^2+\frac{k}{2}\sG^2$ on
the first line of formula~\eqref{TQtilde} for $\Qt_4$ are 
conformally invariant.  So~\eqref{tildetransform} still holds if they are
omitted.  However, they are needed to obtain the factorization in the case
of minimal submanifolds of Einstein manifolds.   
\end{remark}

We next prove Theorem~\ref{gjmsmain} in 
the remaining case (1c) for $\ell = k/2+1$.  The following lemma is the 
key.  It generalizes to higher order the cancellation of the occurrences of
$U_{(4)}$ in $K_\al{}^\al$ and $F$ in the formula~\eqref{traceh4} for
$\tr \htt_4$.    

\begin{lemma}\label{Udependence}
Let $Y$ be an extension of $\Si$ described as a normal graph~\eqref{Y}, 
\eqref{Uexpand} with $U_{(2)}=\frac{1}{2}H$.  
Let $m\geq 2$.  When the induced metric $h_+$ on $Y$ is written in normal
form~\eqref{newnf},~\eqref{newh}, the contribution of $U_{(2m)}$ to
$\htt_{2m}$ is $-2\mathring{L}_{\al\be\al'}U^{\al'}_{(2m)}$.    
\end{lemma}
\begin{proof}
First consider the Taylor expansions of $\hb_{\al\be}$, $\hb_{\al 0}$,
$\hb_{00}$ given by~\eqref{hbar}.  The coefficients of orders up to 
$r^{2m}$ in $\hb_{\al\be}$ and $\hb_{00}$ and up to $r^{2m-1}$ in 
$\hb_{\al  0}$ are the only ones that can affect $\htt_{2m}$.  Staring at
\eqref{hbar} and recalling that the $g_{ij}$ are evaluated at
$(x,u(x,r),r)$ and that $g_{\al\al'}=O(r^2)$ and $u^{\al'}=O(r^2)$, one
sees that for $j<2m$, $U_{(2m)}$ cannot enter into the coefficient of
$r^j$ in any of $\hb_{\al\be}$, $\hb_{\al 0}$, or $\hb_{00}$.  So we are 
left with determining the contribution of $U_{(2m)}$ to the $r^{2m}$
coefficient in $\hb_{\al\be}$ and $\hb_{00}$.   

It is clear that $U_{(2m)}$ cannot
contribute to the $r^{2m}$ coefficient of the second or third terms on the
right-hand side in the formula for $\hb_{\al\be}$ in~\eqref{hbar}.  To
obtain the expansion in $r$ of the first term $g_{\al\be}(x,u(x,r),r)$,
first Taylor expand $g_{\al\be}(x,u,r)$ in $r$ and then expand in $u$ the
resulting coefficients and substitute for each $u$ the expansion     
of $u(x,r)$ in $r$.  It is clear that the $r^{2m}$ coefficient in $u(x,r)$
can only contribute to the $r^{2m}$ coefficient in the resulting expansion
when $u(x,r)$ is substituted into the first term $g_{\al\be}(x,u,0)$.  And 
since on $\Si$ we have $g_{\al\be,\al'}=-2L_{\al\be\al'}$, it follows that
\[
\begin{split}
g_{\al\be}(x,u(x,r),0) & = g_{\al\be}(x,0,0) +
g_{\al\be,\al'}(x,0,0)u^{\al'}(x,r) +O(u^2)\\
&=h_{\al\be}(x) -2L_{\al\be\al'}(x)u^{\al'}(x,r) +O(u^2).  
\end{split}
\]
So the contribution of
$U_{(2m)}$ to the $r^{2m}$ coefficient in $\hb_{\al\be}$ is
$-2L_{\al\be\al'}U_{(2m)}^{\al'}$.  For $\hb_{00}$, substitute
$u^{\al'}{}_{,r}=H^{\al'}r+\cdots + 2m u_{(2m)}^{\al'}r^{2m-1}$
into~\eqref{hbar} to see that the contribution of $U_{(2m)}$ to the
$r^{2m}$ coefficient of $\hb_{00}$ is
$4m H_{\al'}U_{(2m)}^{\al'}$.
Set $\la = 4m H_{\al'}U_{(2m)}^{\al'}$.  

Now we have to transform $h_+$ to normal form as in Lemma~\ref{nflemma}.
Let $h_+^{(0)}$ be the metric obtained by  
truncating the expansion of $U_r$ at order $2m-2$, i.e.\ by using  
$U_r=U_{(2)}r^2+\cdots+U_{(2m-2)}r^{(2m-2)}$.
The above identification of the contribution of $U_{(2m)}$
shows that $h_+ = h_+^{(0)} + r^{-2}\kb$, where
\[
\begin{aligned}
\kb_{\al\be}&=-2L_{\al\be\al'}U^{\al'}_{(2m)}r^{2m}+O(r^{2m+2}),\\ 
\kb_{\al 0} &= O(r^{2m+1}),\\
\kb_{00}&= \la r^{2m} +O(r^{2m+2}).
\end{aligned}
\]
Let $\varphi$ be a diffeomorphism which restricts to the identity on $\Si$ and
satisfies $\varphi^*r=r+O(r^2)$, and for which $h_+^{(1)}:=\varphi^*h_+^{(0)}$ is
in normal form.  It follows that 
$\varphi^*h_+=h_+^{(1)}+\varphi^*\left(r^{-2}\kb\right)=r^{-2}\hb_\varphi$, 
with $\hb_\varphi$ of the form
\[
\begin{aligned}
(\hb_\varphi)_{\al\be}&=\hb^{(1)}_{\al\be}-2L_{\al\be\al'}U^{\al'}_{(2m)}r^{2m}+O(r^{2m+2}),\\ 
(\hb_\varphi)_{\al 0} &= O(r^{2m+1}),\\
(\hb_\varphi)_{00}&= 1+\la r^{2m} +O(r^{2m+2}).
\end{aligned}
\]
It is easily verified that the the coordinate change  
$x^{\al}=\xt^{\al},\,\, r=\rt -\frac{\la}{4m}\, \rt^{2m+1}$
transforms
$\varphi^*h_+$ to normal form to the same order, with  
\[
\htt_{\rt}
= \Big(\hb^{(1)}_{\al\be} -2L_{\al\be\al'}U^{\al'}_{(2m)}\rt^{2m}
+\frac{\la}{2m}h_{\al\be}\rt^{2m}+O(\rt^{2m+2})\Big)d\xt^\al d\xt^\be. 
\]
Since $\frac{\la}{2m}=2 H_{\al'}U_{(2m)}^{\al'}$, it follows that
the contribution of $U_{(2m)}$ to $\htt_{2m}$ is 
\[
-2L_{\al\be\al'}U_{(2m)}^{\al'} +2 H_{\al'}U_{(2m)}^{\al'}h_{\al\be}
=-2\mathring{L}_{\al\be\al'}U_{(2m)}^{\al'}.  \qedhere
\]
\end{proof}

\begin{remark}
In the more general context in which odd powers and/or log
terms are allowed in the expansion of $U_r$, the
argument of the proof of 
Lemma~\ref{Udependence} also identifies the contribution of these terms to
the expansion of $h_+$ in normal form.
\end{remark}

\bigskip
\noindent
{\it Proof of Theorem~\ref{gjmsmain} in case (1c) for 
  $\ell=k/2+1$.}
As in the other cases, it suffices to show that for $k$ even 
(and $n>k+2$ if $n$ is even), the operator $P_{k+2}$ depends only on the
determined coefficients $U_{(2j)}$ for $2j\leq k$, and not on the
undetermined coefficient $U_{(k+2)}$.

Let~\eqref{newh} be the expansion when $h_+$ is written in normal form
\eqref{newnf}.  Lemma~\ref{Qtrace} shows that $P_{k+2}$ depends on
$\htt_{k+2}$ only through $\tr \htt_{k+2}$.  And Lemma~\ref{Udependence}
shows that $U_{k+2}$ drops out when calculating $\tr \htt_{k+2}$, so that
$\tr \htt_{k+2}$ depends only on the $U_{(2j)}$ for $2j\leq k$, as desired.  

\stopthm

\section{Ambient Realization}\label{ambientsection}  

In this section we show how to reformulate the minimal extension $Y$ in
terms of the ambient space and use this to prove
Theorem~\ref{extensioninvariance} for $k$ even.

We first recall the ambient space and review the relationship between 
ambient and Poincar\'e metrics.  See~\cite{FG2} for details.  The metric  
bundle of the conformal manifold $(M,[g])$ is
$\cG=\{(p,t^2g(p)):p\in M, t>0\}\subset S^2T^*M$ and the ambient  
space is $\cGt:=\cG \times (-\ep,\ep)_\rho$ for $\ep>0$ small.  The choice
of representative metric $g$ induces an identification of $\cGt$ with  
$\R_+ \times M\times (-\ep,\ep)$, points of which we denote $(t,p,\rho)$.  
The ambient space admits dilations $\de_s:\cGt\rightarrow \cGt$ for $s>0$
defined by $\de_s(t,p,\rho)=(st,p,\rho)$.  A 
straight ambient metric in normal form relative to $g$ is a Lorentzian
metric on $\cGt$ of the form  
\begin{equation}\label{ambientmetric}
  \gt=2\rho dt^2+2tdtd\rho+t^2 \sg_\rho
\end{equation}
that asymptotically solves $\Ric(\gt)=0$ to an appropriate order which
depends on the dimension.  Here $\sg_\rho$ is a smooth one-parameter family
of metrics on $M$ with $\sg_0=g$.  The metric $\gt$  
is homogeneous of degree 2 with respect to the dilations $\de_s$: 
$\de_s^*\gt=s^2\gt$.

Let $T=\frac{d}{ds} \de_s|_{s=1}$ denote  
the infinitesimal generator of the dilations, so that $T=t\pa_t$ in the
identification $\cGt=\R_+ \times M\times (-\ep,\ep)$.  Note that
$\|T\|^2_{\gt}=2\rho t^2$.  Introduce a new variable $s>0$ on 
$\{\rho\leq 0\}$ by $s^2=-\|T\|^2_{\gt}$, so that $s=rt$ where 
$r=\sqrt{-2\rho}$.  In the new variables $(s,p,r)$, $\gt$ becomes    
\begin{equation}\label{grelation}
\gt = s^2g_+ -ds^2.
\end{equation}
One consequence of this relation is that the asymptotic conditions
$\Ric(\gt)=0$ and $\Ric(g_+)=-ng_+$ are equivalent.

Define the projection $\pi_X:\cGt\rightarrow X=M\times [0,\ep_0)_r$ by  
$\pi_X(t,p,\rho) = (p, r)$, $r=\sqrt{-2\rho}$.   
Define the hypersurface $\cH:=\{\|T\|^2_{\gt}=-1\}=\{s=1\}\subset \cGt$ and  
denote by $i:\cH\rightarrow \cGt$ the inclusion.  If we identify $\cH$
with $\mathring{X}$ via $\pi_X$, then clearly~\eqref{grelation} implies
$g_+ = i^*\gt$. 

Let $Y$ be a submanifold of $X$ with $1\leq \dim Y\leq \dim X -1$.  Define
the   {\it homogeneous lift} $\Yb$ of $Y$ by $\Yb=\pi_X^*(Y)$, which we can 
identify with $\R_+\times Y$ under our product decomposition
$\R_+ \times M\times (-\ep,\ep)$ and the change of variable
$r=\sqrt{-2\rho}$.  Clearly $\Yb$ 
is  invariant under the dilations $\de_s$.  Observe from~\eqref{grelation}
that $\gt|_{T\Yb}$ has Lorentzian signature.

\begin{proposition}\label{minequiv}
A submanifold $Y\subset \mathring{X}=M\times (0,\ep_0)$ is minimal with  
respect to $g_+$ if and only if its homogeneous lift $\Yb$ is minimal with
respect to $\gt$.
\end{proposition}
\begin{proof}
  Observe that~\eqref{grelation} exhibits $\gt$ as a warped product of the
  form $A(s)g_+ -ds^2$.  Thus the conclusion follows from the same
  observation that we made in \S\ref{gjmssection} to identify the minimal
  extension of a   minimal submanifold of an Einstein manifold:  $Y$ is
  minimal in $\mathring{X}$ with respect to $g_+$ if and only if 
  $\Yb = \R_+\times Y$ is minimal in $\cGt$ with respect to  
  any warped product of this form.
\end{proof}

We mentioned in the introduction that in the case $k=1$, Fine
and Herfray in~\cite{FiH} use the even minimal extension $Y$ of a curve
$\Si\subset M$ to  
study conformal geodesics and conformal canonical parametrizations of
$\Si$.
In order to relate the construction to tractors, they also analyze
the unique lift of $\mathring{Y}$ to a surface contained in the
hypersurface $\cH\subset \cGt$.  But they do not consider the homogeneous
lift of $Y$.     

Proposition~\ref{minequiv} shows that the problem of extending $\Si$ to a
minimal submanifold $Y$ is equivalent to the problem of extending its
homogeneous lift $\overline{\Sigma}$ to a homogeneous minimal submanifold
$\Yb$ of $\cGt$.  Therefore Proposition~\ref{U} gives asymptotic solutions
of the latter problem.

\begin{remark}
As mentioned in the introduction, the GJMS operators were
originally constructed in \cite{GJMS} in terms of the ambient metric, and
it was shown in \cite{GZ} that the construction is equivalent to the one
used here in terms of the Poincar\'e metric.  The equivalence argument in 
\cite{GZ} applies in the general setting of an asymptotically hyperbolic
metric $h_+$ and the ambient-like metric it determines upon replacing $g_+$ 
by $h_+$ in \eqref{grelation}.  Consequently, our extrinsic minimal 
submanifold GJMS operators can also be obtained from either of the two
constructions given in \cite{GJMS} applied to the metric on the
homogeneous minimal extension $\Yb$ that is induced   
from the ambient metric $\gt$ associated to $(M,g)$.  
\end{remark}

Now we turn to the proof of
Theorem~\ref{extensioninvariance}.  The proof follows the same general
outline as the 
proof in Chapter 7 of~\cite{FG2} of the diffeomorphism invariance of the
canonical ambient metric associated to an Einstein metric.  We have to
develop a number of ingredients before we can present the proof at the end
of the section.

Let $\Si^k\subset (M^n,g)$ and let $(x^\al,u^{\al'})$ be adapted
coordinates near $\Si$ defined via the normal exponential map, as
constructed in \S\ref{background}.
We will need the following lemma, which makes explicit certain   
consequences of the fact that $\Si$ is minimal for two conformally 
related Einstein metrics.   

\begin{lemma}\label{2omega}
Suppose $g$ and $\gh=e^{2\om}g$ are conformally related Einstein metrics on
$M$.  Suppose $\Si\subset M$ is minimal for both $g$
and $\gh$.  For $s\geq 0$, let $\na^sL$ be the section of
$\otimes^{s+2}T^*\Si\otimes N\Si$ which is the $s$-th covariant  
derivative of the second fundamental form of $\Si$ with respect to $g$.
Then $(\grad_b\om\,\into)^{s+1}(\na^sL)=0$, where
$(\grad_b\om\into)^{s+1}(\na^sL)$ denotes the contraction of the tangential
gradient of $\om|_\Si$ into any $s+1$ of the $s+2$ covariant indices of 
$\na^sL$. 
\end{lemma}
\begin{proof}
Recall that the mean curvature vector transforms by
$e^{2\om}\widehat{H}^{\al'} = H^{\al'} -\om^{\al'}$.  Since $\Si$ is 
minimal for both $g$ and $\gh$, it follows that $\om^{\al'}=0$ on $\Si$.    

The conformal transformation law for the Schouten tensor reads
\begin{equation}\label{Ptrans}
\widehat{\mathsf{P}}_{ij}=\mathsf{P}_{ij}-\om_{ij} +\om_i\om_j -\tfrac12
\om_k \om^k g_{ij}, 
\end{equation}
where $\om_{ij}={}^g\na^2_{ij}\om$.  
Take the $\al\al'$ component of~\eqref{Ptrans} on $\Si$.  Since
$\widehat{\mathsf{P}}_{ij}$
and $\mathsf{P}_{ij}$ are both multiples of $g_{ij}$ and $\om_{\al'} =0$,
this component reduces to simply $\om_{\al\al'}=0$.  But
$\om_{\al\al'}=\pa_{\al}\om_{\al'} -\Ga_{\al\al'}^i \om_i$.  Now 
$\pa_{\al}\om_{\al'}=0$ since $\om_{\al'}=0$ on $\Si$, and
$\Ga_{\al\al'}^{\be}=\frac12 g^{\be\ga}\pa_{\al'}g_{\al\ga}=-L_{\al\al'}^\be$.
Therefore  $L_{\al\al'}^\be\om_\be=0$, which is equivalent to
$L^{\al'}_{\al\be}\om^\be=0$.  This proves the case $s=0$.

We proceed by induction on $s$.  First consider a case with $s=1$, which is
indicative of the general argument.  Write
\begin{equation}\label{r=1}
L^{\al'}_{\al\be;\ga}\om^\al\om^\ga
=\left(L^{\al'}_{\al\be}\om^\al\om^\ga\right){}_{;\ga}
-L^{\al'}_{\al\be}\om^\al{}_\ga\om^\ga-L^{\al'}_{\al\be}\om^\al\om^\ga{}_\ga.
\end{equation}
The first and last terms on the right-hand side vanish by the case
$s=0$.  For the middle term, take the $\al\ga$ component of
\eqref{Ptrans} and use as 
above that $\mathsf{P}_{ij}$ and $\widehat{\mathsf{P}}_{ij}$ are multiples
of $g_{ij}$ to obtain $\om_{\al\ga}=\om_\al\om_\ga + fg_{\al\ga}$ for some
function $f$ on $\Si$.  In this equation, $\om_{\al\ga}={}^g\na^2_{\al\ga}\om$, whereas in
\eqref{r=1}, $\om^\al{}_\ga$ refers to the 
covariant derivative of $\grad_b\om$ with 
respect to the induced connection.  But these agree since $\om_{\al'}=0$ on
$\Si$.  Substituting 
$\om^\al{}_{\ga}=\om^\al\om_\ga + f\delta^\al{}_{\ga}$ into~\eqref{r=1} 
shows that the right-hand side of~\eqref{r=1} vanishes by the case $s=0$ as
desired.    
The argument for the other $s=1$ case, where $\om^\ga$ is replaced by 
$\om^\be$ in the left-hand side of~\eqref{r=1}, is easier:  after factoring  
out the covariant derivative index $\ga$, all terms on the right-hand side
vanish by the $s=0$ case.

The general inductive
step going from $s$ to $s+1$ is similar.  One factors out the last
covariant derivative index on $\na^{s+1}L$ at the expense of terms
involving $\na^2\om$.  If the last covariant derivative index is the free
index, the induction hypothesis immediately implies the result.  If the
last covariant derivative index is one of the contracted indices,
substitute the tangential component of~\eqref{Ptrans} for all the second 
derivative terms.  It is easily seen that all terms vanish by the 
induction hypothesis. 
\end{proof}

If $g$ is Einstein with 
$\Ric(g)=\la(n-1)g$, then the canonical ambient metric is given by 
\eqref{ambientmetric} with $\sg_\rho = (1+\tfrac12 \la\rho)^2g$.  If $\Si$
is minimal, then $\Yb=\R_+\times \Si\times (-\ep,0]$ is a  
homogeneous minimal extension of 
$\R_+\times\Si$, which we call the {\it canonical extension}.  More general
homogeneous extensions are written as a graph
$\Yb=\{(t,x^\al,u^{\al'},\rho):t\in \R_+, u^{\al'}=u^{\al'}(x,\rho)\}$;
the canonical extension is given by $u^{\al'}=0$.  For comparison, the  
graphing function in the 
Poincar\'e metric picture is $u^{\al'}=u^{\al'}(x,-r^2/2)$.  In particular,
the indeterminacy in $u^{\al'}$ appears at order $\rho^{k/2+1}$.  Let 
$\htt=\iota^*\gt$ be the induced metric, where $\iota:\Yb\rightarrow\cGt$. 
We observed above that $\htt$ has Lorentzian signature for $\rho$ small.
This will be explicit in~\eqref{inducedmetric} below.  Write  
\[
g=g_{\al\be}dx^\al dx^\be+2g_{\al\al'}dx^\al du^{\al'}
+g_{\al'\be'}du^{\al'}du^{\be'}, 
\]
where all $g_{ij}$ are functions of $(x,u)$ and $g_{\al\al'}=0$ at $u=0$.
Since $g$ is Einstein, $\sg_\rho = (1+\tfrac12 \la\rho)^2g$.   
Now
\[
\iota^*\sg_{\rho}=k_{\al\be}dx^\al dx^\be +2k_{\al\infty}dx^\al d\rho 
+k_{\infty\infty}d\rho^2,
\]
where
\begin{align*}
k_{\al\be}&=(1+\tfrac12 \la\rho)^2
\big(g_{\al\be}+2g_{\al'(\al}u^{\al'}{}_{,\be)} +
g_{\al'\be'}u^{\al'}{}_{,\al}u^{\be'}{}_{,\be}\big) \, ,\\
k_{\al\infty}&=(1+\tfrac12 \la\rho)^2 \big(g_{\al\al'}u^{\al'}{}_{,\rho}
+g_{\al'\be'}u^{\al'}{}_{,\al}u^{\be'}{}_{,\rho}\big)
\, ,
\\
k_{\infty\infty}&=(1+\tfrac12
\la\rho^2)\,g_{\al'\be'}u^{\al'}{}_{,\rho}u^{\be'}{}_{,\rho}\, .   
\end{align*}
The coefficients $k_{\al\be}$, $k_{\al\infty}$, and $k_{\infty\infty}$ are
functions of $(x,\rho)$, and the $g_{ij}$ are evaluated at
$(x,u(x,\rho))$.  
It follows that 
\begin{equation}\label{inducedmetric}
\left(\htt_{\cI\cJ}\right)=
\begin{pmatrix}
  2\rho&0&t\\
  0&t^2k_{\al\be}&t^2k_{\al\infty}\\
  t&t^2k_{\al\infty}&t^2k_{\infty\infty}
\end{pmatrix}
\end{equation}
in $(t,x,\rho)$ coordinates.  We use $0$ for the $t$ direction and $\infty$ 
for the $\rho$ direction. Indices $\cI$, $\cJ$, $\cK$, $\cL$ run over   
$0,\alpha,\infty$, i.e. $0,1,\ldots, k$, $\infty$, and in the following,
indices $I$, $J$, $K$ run over $0,i,\infty$, i.e. $0,1,\ldots, n$,
$\infty$.

Set $e_{\cI}= \pa_\cI +u^{\al'}{}_{,\cI}\pa_{\al'}$, so that $\{e_\cI\}$ is
a basis for $T\Yb$.  Let ${}^\top$, ${}^\perp$ denote the orthogonal
projections with respect to $\gt$ of 
$T\cGt|_{\Yb}$ onto $T\Yb$, $N\Yb$, resp.  Set $e_{\al'}=\pa_{\al'}^\perp$.
Let $\nb$ denote the induced connections on $T\Yb$ and $N\Yb$ and let
$\Gb_{\cI\cJ}^\cK$,  $\Gb_{\cI\be'}^{\al'}$ denote the Christoffel symbols
of $\nb$ on $T\Yb$, $N\Yb$, resp.  Let $\Lb^{\al'}_{\cI\cJ}$ denote  
the components of the second fundamental form of $\Yb$ with respect to the
frames $e_\cI$, $e_{\al'}$.  The defining relations for these quantities
are: 
\[
\left(\nt_{e_\cI}e_\cJ\right)^\top =
\nb_{e_\cI}e_\cJ=\Gb^{\cK}_{\cI\cJ}e_{\cK},\quad
\left(\nt_{e_\cI}e_\cJ\right)^\perp = \Lb^{\al'}_{\cI\cJ}e_{\al'},\quad
\left(\nt_{e_\cI}e_{\be'}\right)^\perp = \nb_{e_\cI}e_{\be'} =
\Gb^{\al'}_{\cI\be'}e_{\al'}.
\]

The Christoffel symbols $\Gt_{IJ}^K$ of $\gt$ with respect to a coordinate
frame $\{\pa_t,\pa_i,\pa_\rho\}$ are given by (3.16) of
\cite{FG2}.  In the case that $g$ is Einstein and
$\sg_\rho =(1+\tfrac12 \la\rho)^2g$, these become

\begin{gather}\label{Christgt}
\begin{gathered}
  \Gt_{IJ}^0=
  \begin{pmatrix}
    0&0&0\\
    0&-\tfrac12 t\la(1+\tfrac12 \la\rho)g_{ij}&0\\
    0&0&0
  \end{pmatrix}\, ,\\
  \Gt_{IJ}^k=
  \begin{pmatrix}
    0&t^{-1}\de_j^k&0\\
   t^{-1}\de_i^k&\Ga^k_{ij}&\frac{\la}{2}(1+\tfrac12
   \la\rho)^{-1}\de_i^k\\ 
    0&\frac{\la}{2}(1+\tfrac12 \la\rho)^{-1}\de_j^k&0 
  \end{pmatrix}\, ,\\
  \Gt_{IJ}^\infty=
  \begin{pmatrix}
    0&0&t^{-1}\\
    0&(-1+\tfrac14 \la^2\rho^2)g_{ij}&0\\
    t^{-1}&0&0
  \end{pmatrix}.\\
\end{gathered}
\end{gather}
In the middle equation, $\Ga_{ij}^k$ denotes the Christoffel symbol of $g$
on $M$.  

If $u^{\al'}=0$, then $e_\cI = \pa_\cI$ and $e_{\al'}=\pa_{\al'}$, so
$\Gb_{\cI\cJ}^\cK = \Gt_{\cI\cJ}^\cK$,
$\Lb^{\al'}_{\cI\cJ}= \Gat^{\al'}_{\cI\cJ}$, and
$\Gb_{\cI\al'}^{\be'} = \Gt_{\cI\al'}^{\be'}$.   
Equation~\eqref{Christgt} gives
\begin{equation}\label{L}
\Lb^{\al'}_{\cI\cJ}=
\begin{pmatrix}
  0&0&0\\
  0&\Ga^{\al'}_{\al\be}&0\\
  0&0&0
\end{pmatrix}=
\begin{pmatrix}
  0&0&0\\
  0&L^{\al'}_{\al\be}&0\\
  0&0&0
\end{pmatrix},
\end{equation}
where $\Ga^{\al'}_{\al\be}$ are the Christoffel symbols for the metric $g$
on $M$, which is independent of $\rho$ (and $t$), and 
$L^{\al'}_{\al\be}=\Ga^{\al'}_{\al\be}$ is the second fundamental form of 
$\Si\subset M$.

Equation~\eqref{L} evaluates $\Lb$ everywhere on $\Yb$ for the canonical
extension $u^{\al'}=0$.  Next we derive the form of the covariant
derivatives of $\Lb$.  We begin with components involving a $0$ index.
These relations depend only on the homogeneity, so they hold for any
homogeneous extension $\Yb$ of $\Si$.   
Recall that $T=d/ds|_{s=1}\delta_s$.  This is a 
vector field on $\cGt$ which is tangent to $\Yb$.  The
components of $T$ are given by $T^I = t\delta^I_0$.  It is shown in
\cite[Proposition 3.4]{FG2} that $\nt T = Id$ is the identity endomorphism
of $T\cGt$ for any ambient metric of the form~\eqref{ambientmetric}, with
$\sg_\rho$ an arbitrary 1-parameter family of metrics on $M$.  
\begin{proposition}\label{0components}
  The covariant derivatives of the second fundamental form of any
  homogeneous $\Yb$ satisfy
  \begin{enumerate}
  \item
    $T^\cJ\Lb^{\al'}_{\cI\cJ;\cK_1\cdots\cK_r}
    =-\sum_{s=1}^r \Lb^{\al'}_{\cI\cK_s;\cK_1\cdots\wh{\cK_s}\cdots\cK_r},$ 
  \item
    $T^\cL\Lb^{\al'}_{\cI\cJ;\cK_1\cdots\cK_s\cL\cK_{s+1}\cdots\cK_r}
    =-(s+1)\Lb^{\al'}_{\cI\cJ;\cK_1\cdots\cK_r}
    -\sum_{t=s+1}^r\Lb^{\al'}_{\cI\cJ;\cK_1\cdots\cK_s\cK_t\cK_{s+1}\cdots\wh{\cK_t}\cdots\cK_r}$. 
\end{enumerate}
\end{proposition}
\noindent
Condition (1) in the case $r=0$ is interpreted as the statement
$T^\cJ\Lb^{\al'}_{\cI\cJ}=0$, or equivalently $\Lb^{\al'}_{\cI 0}=0$.  Note 
that the case $s=r$ in (2) reduces to
\begin{equation}\label{s=r}
T^\cL\Lb^{\al'}_{\cI\cJ;\cK_1\cdots\cK_r\cL}
=-(r+1)\Lb^{\al'}_{\cI\cJ;\cK_1\cdots\cK_r}.  
\end{equation}
\begin{proof}
If $V$ is a vector field tangent to $\Yb$, then $\nt_V T=V$ is also tangent 
to $\Yb$.  So $\Lb(T,V)=0$.  In terms of indices, this is written 
$T^{\cI}\Lb^{\al'}_{\cI \cJ}=0$.  Now differentiate successively using 
$T^{\cI}{}_{;\cJ}=\delta^\cI{}_\cJ$ to obtain (1). 

For (2), we first prove the special case~\eqref{s=r}.  Recall the identity
\begin{equation}\label{Liederiv}
\na_XU = \cL_XU +(\na X).U
\end{equation}
for a torsion free connection $\na$ on a manifold, where $X$ is a vector
field and $U$ is a tensor field, $\cL_X$ denotes the Lie derivative, and 
$(\na X).U$ denotes the algebraic action of the endomorphism $\na X$ on
$U$.  Apply this on $\cGt$, taking $\na=\nt$, $X=T$, and $U$ equal to an
extension of $\nb^r\Lb$ to a tensor field on $\cGt$ near $\Yb$.  Since
$\Lb$ is homogeneous of degree $0$ with respect to the dilations
$\delta_s$, we can choose $U$ to have the same homogeneity so that
$\cL_TU=0$.  The identity endomorphism acts on a tensor 
covariant in $\ell$ indices and contravariant in $m$ indices by
multiplication by $m-\ell$.  So~\eqref{Liederiv} becomes
$\nt_TU = -(r+1)U$.  This implies that the restriction of $\nt_TU$ to $\Yb$
is also a section of $\otimes^{r+2}T^*\Yb\otimes N\Yb$, so that
$\nt_TU|_{\Yb}=\nb_T(\nb^r\Lb)$.  This proves~\eqref{s=r}.  

Now replace $r$ by $s$ in~\eqref{s=r} and differentiate $r-s$ more
times using again $T^{\cI}{}_{;\cJ}=\delta^\cI{}_\cJ$ to obtain (2).   
\end{proof}

Now we restrict to the case $\Yb= \R_+\times \Si\times (-\ep,0]$, which 
is the canonical extension when $g$ is Einstein and $\Si$
is minimal.  At each point $p\in\Yb$, we can identify $N_p\Yb$ with
$N_{\pi_\Si(p)}\Si$, where $\pi_\Si:\Yb\rightarrow \Si$ is the projection
induced by this product decomposition.  
Construct tensors on $\Si$ from the covariant derivatives of
$\Lb$ as follows.  Choose an order $r\geq 0$ of covariant  
differentiation.  Divide the set of symbols $\cI\cJ\cK_1\cdots \cK_r$
into three disjoint subsets $\cS_0$, $\cS_\Si$ and $\cS_\infty$
of cardinalities $s_0$, $s_\Si$, $s_\infty$, resp. Set the 
indices in $\cS_0$ equal to $0$, those in $\cS_\infty$ equal to $\infty$,
and let those in $\cS_\Si$ correspond to $\Si$ in the decomposition
$\Yb = \R_+\times \Si\times (-\ep,0]$. (In local coordinates, the indices
in $\cS_\Si$ vary between $1$ and $k$.)  Evaluate the resulting component
$\Lb^{\al'}_{\cI\cJ;\cK_1\cdots\cK_r}$
at $\rho =0$ and $t=1$.  This defines a tensor on $\Si$ which is a section
of $\otimes^{s_{\Si}}T^*\Si\otimes N\Si$, which we denote by
$\Lb^{(r)}_{\cS_0,\cS_\Si,\cS_\infty}$.  

\begin{proposition}\label{Scomponents}
The tensor $\Lb^{(r)}_{\cS_0,\cS_\Si,\cS_\infty}$ is a linear combination
of tensor products of an iterated covariant derivative of the second 
fundamental form $L$ of $\Si$ with some power of the induced metric
$g|_{T\Si}$, with some ordering of the covariant indices.    
\end{proposition}

\noindent
We provide a clarification of the statement.  For any order $s\geq 0$ of 
covariant differentiation and any $m\geq 0$, and any ordering of the
indices, we can form a tensor on $\Si$
of the form $\na^sL\otimes (\otimes^m g|_{T\Si})$.  The statement is that for 
any $\cS_0$, $\cS_\Si$, $\cS_\infty$, the tensor
$\Lb^{(r)}_{\cS_0,\cS_\Si,\cS_\infty}$ is a linear combination of tensors
of this form.  Since the ranks must agree, only terms for which  
$s+2+2m=s_\Si$ can appear in the linear combination.   
\begin{corollary}\label{Lvanishing}
$\Lb^{(r)}_{\cS_0,\cS_\Si,\cS_\infty}=0$ if $s_\Si=0$ or $1$.  
\end{corollary}
\begin{proof}
$s_\Si=s+2+2m\geq 2$.
\end{proof}

\noindent
{\it Proof of Proposition~\ref{Scomponents}.}
We prove a stronger statement:  Denote  
by $\Lb^{(r)}_{\cS_0,\cS_\Si,\cS_\infty}(\rho)$ the 1-parameter family of
tensors on $\Si$ constructed exactly the same way, but not restricting to
$\rho =0$.  So
$\Lb^{(r)}_{\cS_0,\cS_\Si,\cS_\infty}=\Lb^{(r)}_{\cS_0,\cS_\Si,\cS_\infty}(0)$.  
We prove by induction on $r\geq 0$ that for each choice of $\cS_0$,
$\cS_\Si$, $\cS_\infty$ with $s_0+s_\Si+s_\infty = r+2$, the tensor  
$\Lb^{(r)}_{\cS_0,\cS_\Si,\cS_\infty}(\rho)$  is a linear combination of
tensors of the form $f(\rho)\na^sL\otimes (\otimes^m g|_{T\Si})$ with some 
ordering of the indices of each term, where $f(\rho)$ is a smooth function
near $\rho =0$ which can vary from term to term.  The desired statement
follows upon setting $\rho =0$.

The case $r=0$ is clear from~\eqref{L}:  the only nonzero component of
$\Lb$ is $\Lb_{\cI\cJ}^{\al'} = L_{\cI\cJ}^{\al'}$.  For this component,
$s=m=0$ and $f(\rho)=1$.  

Assume the induction statement is true for $r$.  Consider a tensor
$\Lb^{(r+1)}_{\cS_0,\cS_\Si,\cS_\infty}(\rho)$, whose components we write 
$\Lb^{\al'}_{\cK_1\cK_2;\cK_3\cdots\cK_{r+2}\cI}$.
Proposition~\ref{0components} shows that a zero index can be removed at the
expense of commuting the remaining indices.  So we can assume that
$\cS_0=\emptyset$.  Write
\begin{multline*}
\Lb^{\al'}_{\cK_1\cK_2;\cK_3\cdots\cK_{r+2}\cI}
=\pa_\cI\Lb^{\al'}_{\cK_1\cK_2;\cK_3\cdots\cK_{r+2}}
-\Gb^\cJ_{\cI\cK_1}\Lb^{\al'}_{\cJ\cK_2;\cK_3\cdots\cK_{r+2}}
-\Gb^\cJ_{\cI\cK_2}\Lb^{\al'}_{\cK_1\cJ;\cK_3\cdots\cK_{r+2}}\\
-\blue{\cdots}
-\Gb^\cJ_{\cI\cK_{r+2}}\Lb^{\al'}_{\cK_1\cK_2;\cK_3\cdots\cK_{r+1}\cJ} 
+\Gb^{\al'}_{\cI\be'}\Lb^{\be'}_{\cK_1\cK_2;\cK_3\cdots\cK_{r+2}}.
\end{multline*}
As noted above, for the Christoffel symbols we have 
$\Gb_{\cI\cK}^{\cJ}=\Gt_{\cI\cK}^{\cJ}$ and
$\Gb_{\cI\be'}^{\al'}=\Gt_{\cI\be'}^{\al'}$, and the $\Gt_{IJ}^K$ are given
by~\eqref{Christgt}.

First consider the case $\cI=\infty$, so $\pa_\cI = \pa_\rho$.  For the
first term on the right-hand side, apply the induction hypothesis to
$\Lb^{\al'}_{\cK_1\cK_2;\cK_3\cdots\cK_{r+2}}$.  The $\pa_\rho$ just hits
the function $f(\rho)$ in each term in the linear combination, so the first
term on the right-hand side has the desired form.  For the last term,
substitute
$\Gb_{\infty\be'}^{\al'}
= \frac{\la}{2}(1+\tfrac12 \la\rho)^{-1}\de_{\be'}^{\al'}$  
and apply the induction hypothesis to see that it also has the desired
form.  Since $\Gt_{\infty\infty}^\cJ=0$ for all $\cJ$, the terms with a
factor $\Gb^\cJ_{\cI\cK_i}$ for which $\cK_i=\infty$ all vanish.  The
only other terms involve $\Gb^\cJ_{\cI\cK_i}$ with $\cK_i\in \cS_\Si$.
For these terms, $\Gb^\cJ_{\cI\cK_i}=0$ unless $1\leq \cJ\leq k$, in which
case
$\Gb^\cJ_{\cI\cK_i} =\frac{\la}{2}(1+\tfrac12
\la\rho)^{-1}\de_{\cK_i}^{\cJ}$. 
So again, the induction hypothesis implies these terms have the desired
form. 

Finally consider the case $\cI\in \cS_\Si$.  Recall that
$\Gb_{\cI\infty}^{\cJ}=0$ unless $1\leq\cJ\leq k$, in which case 
$\Gb_{\cI\infty}^{\cJ}=\frac{\la}{2}(1+\tfrac12\la\rho)^{-1}\de_{\cI}^{\cJ}$.
So the terms on the right-hand side with a factor $\Gb_{\cI\cK_i}^{\cJ}=0$
with $\cK_i=\infty$ all are of the desired form.  If $\cK_i\in \cS_\Si$,
then $\Gb_{\cI\cK_i}^{\cJ}$ is nonvanishing for all $\cJ$.  Since 
$\Gb_{\cI\cK}^0=-\tfrac12 t\la(1+\tfrac12 \la\rho)g_{\cI\cK}$, we can apply
the induction hypothesis to conclude that the terms with $\cJ=0$ are of the 
desired form (recall that $t=1$ in the definition of
$\Lb^{(r+1)}_{\cS_0,\cS_\Si;\cS_\infty}(\rho)$).  Since
$\Gb_{\cI\cK}^\infty=(-1+\tfrac14 \la^2\rho^2)g_{\cI\cK}$, we can likewise
apply the induction hypothesis to obtain the desired form for these terms.
This leaves the terms with $\cK_i\in \cS_\Si$ and $1\leq \cJ\leq k$.  In
this case $\Gb_{\cI\cK}^\cJ =\Ga_{\cI\cK}^\cJ$ is the Christoffel symbol
for $g$ on $M$.  Applying the induction hypothesis again, all of these
terms combine with the first and last term to  
produce the covariant derivative on $M$ applied to each of the terms
$\na^sL\otimes(\otimes^mg|_{T\Si})$ in the linear combination.  This gives 
further terms of the desired form.  
\stopthm

Next consider extensions $\Yb$ defined by nonzero $u^{\al'}$.  
\begin{proposition}\label{canchar}
Let $k\geq 2$ be even.  Let $g$ be Einstein and let $\gt$ be given by
\eqref{ambientmetric} with $\sg_\rho=(1+\tfrac12 \la\rho)^2g$.  If
$\Si\subset M$ is minimal, then $\Yb=\R_+\times \Si\times (-\ep,0]$ is to
infinite order the unique homogeneous minimal extension of $\Si$ satisfying   
\begin{equation}\label{tensorialcondition}
\Lb^{\al'}_{\infty\infty;\underbrace{\scriptstyle{\infty\cdots\infty}}_{k/2-1}}\big|_{\rho=0}=0. 
\end{equation}
\end{proposition}
\begin{proof}
  First calculate the leading term in $\Lb^{\al'}_{\infty\infty}$, which
  recall is defined by 
  $\left(\nt_{e_\infty}e_\infty\right)^\perp=\Lb^{\al'}_{\infty\infty}e_{\al'}$. 
  Recalling from~\eqref{Christgt} that $\nt_{\pa_\rho}\pa_\rho = 0$, we
  have 
\[
\begin{split}
\nt_{e_\infty}e_\infty 
&=\nt_{\pa_\rho+u^{\al'}{}_{,\rho}\pa_{\al'}}e_\infty
=\nt_{\pa_\rho}e_\infty +u^{\al'}{}_{,\rho}\nt_{\pa_{\al'}}e_\infty\\
&=\nt_{\pa_\rho}\left(\pa_\rho +u^{\al'}{}_{,\rho}\pa_{\al'}\right)
+u^{\al'}{}_{,\rho}\nt_{\pa_{\al'}}\left(\pa_\rho
+u^{\be'}{}_{,\rho}\pa_{\be'}\right)\\
&=(\pa_\rho^2u^{\al'})\pa_{\al'}
+2u^{\al'}{}_{,\rho}\nt_{\pa_\rho}\pa_{\al'}
+u^{\al'}{}_{,\rho}u^{\be'}{}_{,\rho\al'}\pa_{\be'}
+u^{\al'}{}_{,\rho}u^{\be'}{}_{,\rho}\nt_{\pa_{\al'}}\pa_{\be'}.  
\end{split}
\]
Since $e_{\al'}=(\pa_{\al'})^\perp$, we obtain
$\Lb^{\al'}_{\infty\infty} = \pa_\rho^2u^{\al'}+lots$, where $lots$ 
consists of terms involving fewer $\rho$-derivatives of $u^{\al'}$.

Now successive differentiation shows that 
$\Lb^{\al'}_{\infty\infty;\underbrace{\scriptstyle{\infty\cdots\infty}}_m}=\pa_\rho^{m+2}u^{\al'}+lots$ 
for $m\geq 0$.  In particular, this shows that
$\pa_\rho^{k/2+1}u^{\al'}\big|_{\rho=0}$ is uniquely determined by 
$\Lb^{\al'}_{\infty\infty;\underbrace{\scriptstyle{\infty\cdots\infty}}_{k/2-1}}\big|_{\rho=0}$
and the values of $\pa_\rho^ju^{\al'}\big|_{\rho =0}$ for $j\leq k/2$.   
We know that any minimal extension of $\Si$ satisfies
$u^{\al'}=O(\rho^{k/2+1})$, and $u^{\al'}$ vanishes to infinite order if
and only if $u^{\al'}=O(\rho^{k/2+2})$.  Corollary~\ref{Lvanishing} shows
that the choice $u^{\al'}=0$ makes 
$\Lb^{\al'}_{\infty\infty;\underbrace{\scriptstyle{\infty\cdots\infty}}_{k/2-1}}\big|_{\rho=0}=0$.    
\end{proof}

\noindent
{\it Proof of Theorem~\ref{extensioninvariance}.}
We have already observed that this follows from the uniqueness of the
even minimal extension when $k$ is odd.  So we can assume that $k\geq 2$ is 
even.

We prove the analogous statement for the diffeomorphism relating the
ambient metrics of $g$ and $\gh$.  
Let $\gth$ be the ambient metric on $\cGth=\R_+\times M\times (-\ep,\ep)$   
obtained by replacing $\gt$ by 
$\gth$ and $\sg_\rho$ by
$\widehat{\mathsf{g}}_\rho = (1+\tfrac12 \lh \rho)^2 \gh$
in~\eqref{ambientmetric}.  Let $\chi$ be a homogeneous diffeomorphism mapping
$\cGth$ to $\cGt$ which pulls back $\gt$ to $\gth$ to infinite order at
$\rho =0$ and restricts to the identity on $\cG$.  
Then $\chi$ is uniquely determined to infinite order.  
Let $\Yb\subset \cGt$ and $\Ybh\subset \cGth$ be the canonical extensions
of $\Si$ with respect to $g$ and $\gh$, respectively.
Proposition~\ref{canchar} implies that $\chi^{-1}(\Yb)=\Ybh$ to infinite
order if and only if   
$\chi^{-1}(\Yb)$ satisfies~\eqref{tensorialcondition}.  The condition
\eqref{tensorialcondition} transforms tensorially under the Jacobian of
$\chi$.  This Jacobian is identified at $\rho =0$ in terms of $\om$ 
in (6.8)-(6.10) of~\cite{FG2}.  It follows that 
$\chi^{-1}(\Yb)$ satisfies~\eqref{tensorialcondition} if and only if 
\begin{equation}\label{explicitcondition}
\Lb^{\al'}_{\cI\cJ;\cK_1\cdots\cK_{k/2-1}}
p^\cI{}_\infty p^\cJ{}_\infty p^{\cK_1}{}_\infty \cdots
p^{\cK_{k/2-1}}{}_\infty =0
\end{equation}
at $\rho=0$, where
\[
p^I{}_J=
\begin{pmatrix}
  1&\om_j&-\tfrac12 \om_k \om^k\\
  0&\delta^i{}_j&-\om^i\\
  0&0&1
\end{pmatrix}
\]
and $\Lb^{\al'}_{\cI\cJ;\cK_1\cdots\cK_{k/2-1}}$ refers to the covariant
derivative of $\Lb$ on $\Yb$.  (See the proof of Proposition 6.5 of
\cite{FG2}, and recall that $\om^{\al'}=0$ on $\Si$.) 

Expanding out~\eqref{explicitcondition}, one obtains a linear combination
with smooth coefficients of contractions of  
$\grad_b\om$ with tensors $\Lb^{(r)}_{\cS_0,\cS_\Si,\cS_\infty}$ in which
each index in $\cS_\Si$ is contracted against a factor of $\grad_b\om$.   
Write the tensor $\Lb^{(r)}_{\cS_0,\cS_\Si,\cS_\infty}$ as a linear
combination of tensors of the form $\na^sL\otimes(\otimes^mg|_{T\Si})$ as
in Proposition~\ref{Scomponents}.  Each index of 
each tensor $\na^sL$ which appears is contracted against a factor of
$\grad_b\om$.  Lemma~\ref{2omega} implies that all the terms vanish.    
\stopthm

\end{document}